\documentclass[12pt]{amsart}
\usepackage{graphicx}
\usepackage{amssymb}
\usepackage{epstopdf}
\usepackage{nicefrac}
\usepackage{enumerate}
\usepackage{float}
\usepackage{tikz-cd}
\usepackage{hyperref}
\usepackage{pst-node}
\usepackage{tikz-cd} 
\usetikzlibrary{arrows}

\usepackage[all,cmtip]{xy}

\newtheorem{thm}{thm}[section]
\newtheorem{theorem}[thm]{Theorem}

\newtheorem{question}[thm]{Question}
\newtheorem{cor}[thm]{Corollary}
\newtheorem{defn}[thm]{Definition}
\newtheorem{ex}[thm]{Example}

\newtheorem{lemma}[thm]{Lemma}

\newtheorem{prop}[thm]{Proposition}

\newcommand{\e}{{\epsilon}}

\newcommand{\mQ}{{\mathbb Q}}
\newcommand{\mR}{{\mathbb R}}

\newcommand{\mX}{{\mathbb X}}
\newcommand{\mY}{{\mathbb Y}}
\newcommand{\mZ}{{\mathbb Z}}

\newcommand{\cA}{{\mathcal A}}
\newcommand{\cB}{{\mathcal B}}
\newcommand{\cC}{{\mathcal C}}

\newcommand{\cG}{{\mathcal G}}
\newcommand{\cF}{{\mathcal F}}
\newcommand{\cH}{{\mathcal H}}

\newcommand{\cM}{{\mathcal M}}
\newcommand{\cN}{{\mathcal N}}

\newcommand{\cP}{{\mathcal P}}

\newcommand{\cR}{{\mathcal R}}
\newcommand{\cS}{{\mathcal S}}
\newcommand{\cT}{{\mathcal T}}

\newcommand{\cW}{{\mathcal W}}

\newcommand{\h}{{\overline{H^+_1}}}
\newcommand{\p}{{\overline{\pi^+_1}}}
\newcommand{\eX}{{\underleftarrow\lim(X_i,f_i,p_i)}}
\newcommand{\eY}{{\underleftarrow\lim(Y_i,g_i,q_i)}}
\newcommand{\cc}{{\mathbf c}}
\newcommand{\s}{{\sim_T\,}}
\newcommand{\pps}{PPS}

\usepackage{tikz}

\input xy
\xyoption {all}

\thanks{
The second author thanks Durham University for the period of research leave during which much of this work was done.}
\title{A Complete Invariant for Flow Equivalence}

\date{\today}

\begin{document}

\author{Alex Clark}
\address[Alex Clark]{Centre for Complex Systems, Queen Mary University of London, London, E1 4NS, UK}
\email{alex.clark@qmul.ac.uk}

\author{John Hunton}
\address[John Hunton]{Department of Mathematical Sciences, Durham University, Upper Mountjoy Campus, Stockton Road, Durham, DH1 3LE, UK}
\email{john.hunton@durham.ac.uk}

\maketitle

\begin{abstract} Minimal flow spaces of dimension 1 are among the most fundamental limit sets in dynamical systems.  These invariant sets occur as the typical minimal sets in surface flows, the minimal sets of suspensions of subshifts (for example, in Lorenz template models of the Lorenz attractor) and the hulls of repetitive tilings of dimension one. Here we establish a complete invariant for the flow equivalence of such objects. The invariant takes values in a category of `positive trope'  classes of inverse sequences of free groups and positive maps, or alternatively within a certain category of symbolic systems. Moreover, every such symbolic system is realised by a flow space, and we thus have a one to one correspondence between flow equivalence classes of minimal  flow spaces and positive trope classes of such systems. At the same time, this provides a complete invariant both for flow equivalence of minimal $\mZ$-Cantor dynamical systems and for germinal equivalence of minimal $\mZ$-Cantor  systems.  Our work thus greatly extends that of Barge and Diamond on their complete invariant of primitive substitution tilings, and provides counterpoint to the work of Giordano, Putnam and Skau on orbit equivalence of $\mZ$-Cantor systems. 
\end{abstract}

\addtocontents{toc}{\protect\thispagestyle{empty}}
\tableofcontents
\thispagestyle{empty}

\setcounter{page}{1}

\section{Introduction}

In this work we  consider compact, metrizable spaces $\mX$ of dimension one admitting a continuous action of $\mR$ without fixed points. For brevity, we refer to such an action as a \emph{flow} and to such a space $\mX$ as a \emph{flow space.}  The path components, or \emph{leaves} of a flow space coincide with the orbits of the flow. For a connected flow space, a homeomorphism will either preserve or reverse the orientation of each flow orbit. A homeomorphism that preserves the orientations of orbits is known as a \emph{flow equivalence}. A flow space is \emph{minimal} if every orbit is dense in the whole space; such a space is necessarily connected. A flow space is called \emph{aperiodic} if none of the orbits are periodic, i.e., no path component is homeomorphic to  the circle $S^1$. 

Flow spaces occur as the non-singular minimal sets of general flows on surfaces of genus 2 or more and are ubiquitous in genus 1. In particular, flow spaces occur as the continuous extensions of interval exchange maps.

We present here a complete invariant of minimal flow spaces up to flow equivalence. Along the way we see also a complete invariant of \emph{pointed} minimal flow spaces, that is, of flow spaces $(\mX,x_\bullet)$ with distinguished point $x_\bullet$, where the flow equivalence is required to take distinguished point to distinguished point.

By the results of \cite{KS} there is a direct link between flow spaces and homeomorphisms of zero--dimensional spaces. In particular, any flow space is flow equivalent to the suspension of a homeomorphism of a zero--dimensional space, which will be a Cantor set in the case of a minimal, aperiodic flow space. The key notion that plays the role of flow equivalence for Cantor homeomorphisms is that of  \emph{germinal equivalence}, Definition \ref{germinal}, and  our results  provide a complete invariant for  germinal equivalence of minimal Cantor set homeomorphisms.

We note two earlier sets of results, providing more context for this work. The family of minimal flow spaces includes (but is certainly not limited to) the family of one dimensional, aperiodic tiling spaces, \cite{BG}. In turn this includes the special, and much studied, class of  tilings defined by primitive, non-periodic substitutions. A complete invariant up to flow equivalence for this last class has been provided by Barge and Diamond \cite{BD}. While there are some parallels with our work, the techniques needed for their main result rely heavily on the substitution nature of such spaces, and is not generalisable even to a wider class of tiling spaces. The difficulties arise (both for \cite{BD} and in our work)  from moving from a pointed invariant to an invariant of a general flow equivalence, and we must utilise a significantly different approach.

Secondly, we note the work of Giordano, Putnam and Skau \cite{GPS} and others (especially \cite{HPS, DHS}) who use $C^*$-algebra  $K$-theory to provide a complete invariant of minimal homeomorphisms of the Cantor set $X$ up to \emph{orbit equivalence}. Recall that two homeomorphisms $f\colon X_1\to X_1$ and $g\colon X_2\to X_2$ of the Cantor set are orbit equivalent if there is a homeomorphism $h\colon X_1\to X_2$ with 
$$h\left(\mbox{orbit}_f(x)\right)\ =\ \left(\mbox{orbit}_gh(x)\right)\qquad\mbox{ for all $x\in X_1$.}$$

As noted above, the action of suspending a homeomorphism of a Cantor set gives a flow space, but the notions of orbit equivalence and flow equivalence are not  the same, as demonstrated by Example \ref{twistedfib} and noted earlier by others, e.g., in \cite{BD}. A key issue is that a flow equivalence on a suspension must preserve the order of points in the Cantor set through which the flow line passes, whereas in an orbit equivalence the order may be jumbled. In terms of the invariants for such notions, the $K$-theory is (perhaps ironically given the nomenclature sometimes applied) essentially an \emph{abelian} invariant: in the case of its application to a substitution tiling, for example, it becomes a function of the matrix of the substitution (a homological object) \cite{DHS}. In contrast, to obtain our flow invariant, critical use is made of the non-commutativity of the homotopy tools we apply, in particular that of the fundamental group.

\begin{center}
$*$
\end{center}

\medskip
The invariants  we construct depend crucially on finding appropriate presentations of the underlying flow spaces. We do this by modelling the flow spaces as limits of certain very particular types of inverse sequences, objects which we term \emph{expansions}. The building blocks of these presentations are maps from the flow space to one point unions of finitely many circles that are oriented so as to respect the flow direction; these play a role that is similar to characters in Pontryagin duality. We introduce these and their properties in Section \ref{Sectexpan}, where we show that any minimal flow space has an expansion; we also show that such an expansion can be encoded via a symbolic sequence. Moreover, we prove in Theorems \ref{symbolicexpansion} and \ref{minimal} that \emph{any} symbolic sequence that satisfies two simple combinatorial conditions represents a minimal, aperiodic flow space. There is a loose analogue of this equivalence in the case of the dynamics of a homeomorphism on a Cantor set where the role of combinatorial object is played by the Bratteli diagram, as established in \cite{HPS}.

While it follows immediately that appropriately defined equivalences between expansions give rise to flow equivalences between the flow spaces they represent, the converse is more subtle. In Section \ref{SectRigidity} we show that any flow equivalence between flow spaces induces an equivalence of any pair of flow expansions used to model those spaces. In turn, and in the spirit of certain rigidity theorems for abelian topological groups,  substitution tiling spaces and  tiling spaces more generally, this leads to a \emph{rigidity theorem}, Theorem \ref{isotopic}, for flow spaces; this may be of independent interest.

A further corollary of this part of our work makes precise the relation between flow equivalence and the notion of germinal equivalence mentioned above. Specifically, in Theorem \ref{germ}  we prove that two minimal homeomorphisms of the Cantor set are germinally equivalent if and only if their suspension spaces are flow equivalent.

It is certainly not the case that every aperiodic flow space $\mX$ is the space of an aperiodic tiling, but the tiling spaces are `dense' in the set of all flow spaces in the following sense. Given any expansion of a flow space $\mX$ we obtain a sequence of tiling spaces $\mathcal{T}_n$ with factor maps $\mathcal{T}_m\to \mathcal{T}_n$ for all $m>n$, and a homeomorphism $\mX=\lim\mathcal{T}_n$. Under specified conditions, this homeomorphism is a conjugacy. Specifically, for a large class of minimal flow spaces, including  examples that are not topologically weakly mixing,  the flow on the flow space is the inverse limit of the translation flows on the corresponding tiling spaces. 

A flow space $\mX$ with a given expansion $\mX=\lim X_n$ also defines a distinguished point $x_\bullet\in \mX$, namely the unique point that maps to the wedge points of all the spaces $X_n$. It is useful to record this as a \emph{pointed space} $(\mX,x_\bullet)$. We refer to such a situation by saying that $\{X_n\}$ is \emph{an expansion of $\mX$ about the point $x_\bullet$.} Not only does every minimal flow space have an expansion, but given \emph{any} point in $\mX$ then there are expansions of $\mX$ about that point. (For non-minimal flow spaces there may be expansions about some points, but not about others, an issue explored in more detail in Section \ref{Sectexpan}.) A \emph{pointed flow equivalence} $(\mX,x_\bullet)\to(\mY,y_\bullet)$ is then a flow equivalence $f\colon \mX\to\mY$ that also satisfies $f(x_\bullet)=y_\bullet$.

Our first  complete invariant, that of pointed flow equivalence, is  established in Section \ref{pointedsection}. It is  a functor defined on any pointed flow space $(\mX,x_\bullet)$ which has an expansion about the point $x_\bullet$. It takes values in the category of inverse systems of finitely generated free groups with distinguished generators (\emph{positive generators}) and bonding maps taking the positive words to positive words. For a flow space $(\mX,x_\bullet)$, its image, $\overline{\pi_1^+}(\mX,x_\bullet)$, is given via any  expansion by  taking fundamental groups of the individual terms in the expansion. There is a natural notion of equivalence of such algebraic systems and we show in Theorem \ref{mainpointedtheorem} that if $(\mX,x_\bullet)$  and $(\mY,y_\bullet)$  have expansions about their respective basepoints, then there is a pointed flow equivalence $f\colon(\mX,x_\bullet)\to(\mY,y_\bullet)$ if and only if $\overline{\pi_1^+}(\mX,x_\bullet)$ and $\overline{\pi_1^+}(\mY,y_\bullet)$ are equivalent.

Standard techniques can readily show that if there is a pointed  flow equivalence $f\colon(\mX,x_\bullet)\to(\mY,y_\bullet)$ then there is one $f'\colon(\mX,x'_\bullet)\to(\mY,y'_\bullet)$ for any $x_\bullet', y_\bullet'$ respectively lying on the same leaves as $x_\bullet, y_\bullet$. However, it is not the case that failure of there being a \emph{pointed} flow equivalence $(\mX,x_\bullet)\to(\mY,y_\bullet)$ means that the underlying flow spaces $\mX$ and $\mY$ are not flow equivalent. For example, even in the simple case of a primitive substitution tiling space, it is shown in \cite{BD} that there are a finite number of distinguished leaves (the \emph{asymptotic composants}), and any homeomorphism $\mX\to \mY$ must map the distinguished leaves in $\mX$ to those of $\mY$. Thus the invariant $\overline{\pi_1^+}(\mX,x_\bullet)$ cannot immediately give us a complete invariant of general flow equivalence as it will distinguish  $(\mX,x_\bullet)$ and $(\mX,x'_\bullet)$ merely when $x_\bullet$ and $x_\bullet$ lie on leaves that cannot be mapped to each other.

To obtain the invariant in the unpointed category, more care has to be taken and a construction using conjugacy classes of homomorphisms inspired by the work of Fox \cite{F1,F2,F3} yields an invariant that does not depend on the base point.  In short summary, presented in detail in Section \ref{mainunpointedtheorem}, the method runs as follows. On the one hand we have flow spaces $\mX$, which admit expansions, together with flow equivalences; an expansion of a flow space naturally defines a basepoint $x_\bullet\in\mX$. On the other we have algebraic, or symbolic systems, denoted in Section \ref{mainunpointedtheorem} as \emph{\pps}'s, in which the functor $\overline{\pi_1^+}$ takes values. This target category comes with two notions of equivalence. One of these is effectively that of pro-equivalence in any pro-category; the other is that of  \emph{positive trope} equivalence, following the nomenclature of Fox. It is convenient to define, for a \pps\ $G_\bullet$, its \emph{positive trope class} $PT(G_\bullet)$ as its positive trope equivalence class, and correspondingly we write $PT(\mX)$ for the positive trope class $PT(\overline{\pi_1^+}(\mX,x_\bullet))$ which we prove is independent of the choice of basepoint $x_\bullet$. This $PT$ then becomes our complete invariant for flow equivalences. Our main results may thus be summarised as follows.

\begin{theorem}
\begin{enumerate}
\item Minimal, aperiodic flow spaces $\mX$ and $\mY$ are  flow equivalent if and only $PT(\mX)=PT(\mY)$, i.e., if the \pps's $\overline{\pi_1^+}(\mX,x_\bullet)$ and $\overline{\pi_1^+}(\mY,y_\bullet)$ are positive trope equivalent for any choice of $x_\bullet$ and $y_\bullet$, Corollary \ref{PTmain}.
\item Flow spaces $(\mX,x_\bullet)$ and $(\mY,y_\bullet)$ are pointed flow equivalent if and only if the \pps's $\overline{\pi_1^+}(\mX,x_\bullet)$ and $\overline{\pi_1^+}(\mY,y_\bullet)$ are pro-equivalent, Theorem \ref{mainpointedtheorem}.
\item Any \pps\ is realised via $\overline{\pi_1^+}$, up to pro-equivalence, by a minimal, aperiodic flow space, Theorem \ref{symbolicexpansion} and Theorem \ref{minimal}.
\item Hence there is a one to one correspondence between minimal, aperiodic flow spaces up to flow equivalence and \pps's up to positive trope equivalence.
\end{enumerate}
\end{theorem}

In Section \ref{Applications} we explore some initial applications of our results. In particular, while this last theorem gives us a complete invariant of flow spaces, some simplifications can make it easier to apply, at the cost of  losing its complete nature. One such simplification we make is to exchange the use of the fundamental group with homology, or even cohomology. At the same time, the tiling spaces and subshifts naturally associated to the expansions of flow spaces can be quite tractable, as we see in some specific examples, and so it is to be expected that the $PT$ invariant can be used to classify large families of flow spaces with the appropriate tools from combinatorial group theory.

\section{Background}\label{back}

We collect here briefly some basic background objects and constructions before starting the main work of the paper in the next section.

\subsection{Flow spaces and suspensions.}
By a \emph{flow} on a space $\mX$ we mean a continuous action of $\mR$ on $\mX$, $\phi:\mR \times \mX \to \mX$. The action of $t\in \mR$ on $x \in \mX$ will be denoted $t.\,x,$ and similarly for $S\subset  \mR$ and $A\subset \mX$, we write $S.A=\{\, t.\, x\colon\, t\in S\text{ and }x\in A\}.$ 

A fundamental result for flow spaces of dimension one as considered here is in \cite{KS}, later generalised to the non--compact setting in \cite{AM}. They show 
\begin{theorem}
Any flow space of dimension one $\mX$ is flow equivalent to the suspension of a homeomorphism $h\colon Z \to Z$ of a zero--dimensional compact space $Z.$ 
\end{theorem}
Recall that the \emph{suspension} of a homeomorphism $h\colon Z \to Z$ is the quotient $Z\times [0,1]/\!\!\sim\,$, where $(z,1)\sim (h(z),0)$, and the \emph{suspension flow} is induced by translation in the second factor.  We emphasise that this flow equivalence of the theorem may not be a conjugacy: while the homeomorphism $f\colon \mX\to Z\times [0,1]/\!\!\sim$ can be made to preserve the direction of the flows, it is not necessarily the case that it can be made to satisfy the relation $f(t.\,x)=t.\,f(x)$. 

We also consider a variation on the suspension as above that does not change the topology of the suspension space but can change the nature of the flow.  
The suspension of the homeomorphism $h\colon Z\to Z$ under the \emph{ceiling function} $c\colon Z \to \mR^+$  for a continuous $c$ is the space 
\[
\left\{\,(z,t) \colon 0\leq t \leq c(z)\,,\,z\in Z\right\}/\!\sim\,,
\]
where $(z,c(z))\sim (h(z),0)$ and the flow is again induced by translation.

\subsection{Shifts and tiling spaces.}Important for our investigations will be flows constructed from subshifts of the shift on $n$ symbols $\left( \{a_1,\dots,a_{n} \}^\mZ, s  \right)$. A \emph{tiling space} and its flow is constructed by taking the suspension of a subshift  under a ceiling function which is constant on the cylinder sets $\{\,(x_n)_{n\in \mZ} \colon x_0=a_i\,\}$; see, for example, \cite{SW}. Each point $x$ in such a tiling space is naturally associated to a tiling of $\mR$, identified with the orbit of $x$ and tiled by the intervals above each of the cylinder sets from the construction of the suspension. 

\subsection{Matchbox manifolds and flow spaces.}For our investigations, a \emph{matchbox manifold} is a compact space $X$ for which every point $x\in X$ has an open neighbourhood homeomorphic to $Z_x \times (-1,1)$ for some zero--dimensional space $Z_x$. The subsets of matchbox neighbourhoods of $X$ corresponding to sets of the form $\{z\}\times (-1,1)$ are called  \emph{matches}. 

Flow spaces are examples of matchbox manifolds that are \emph{orientable} in the sense that there are parameterisations (by $\mR$) of the path components which are consistent with a complete atlas of matchbox neighbourhoods  $Z \times (-1,1)$ and the orientation of the matches induced by the projection on the $(-1,1)$ factor. Alternatively, a matchbox manifold is orientable if there is an atlas of matchbox neighbourhoods in which there are continuous choices of orientations on the matches in each neighbourhood, and the orientations agree on the intersections of matchbox neighbourhoods in the atlas, \cite{F}. 

It is shown in \cite{AHO} that every orientable matchbox manifold is a flow space. A matchbox manifold that is not orientable admits a canonical double cover which is an orientable matchbox manifold \cite{F}. Classic examples of unorientable matchbox manifolds are the Plykin attractors in the plane or sphere whose double covers occur as DA attractors in the torus. With the aid of this double covering, our constructions can be easily adapted to invariants of unorientable matchbox manifolds as well.

\subsection{Inverse sequences and  inverse systems.} The invariants we construct take the form of equivalence classes of inverse sequences.  Recall that for a given category $\cC$, an \emph{inverse sequence} is a sequence of objects $C_i$ together with morphisms (bonding maps) $f_i \colon C_{i+1} \to C_i$, which we will denote $(C_i, f_i)$. A composition of bonding maps $f_m\circ \cdots \circ f_n$ will be denoted $f_{n,m}.$

Given two inverse sequences $(C_i,f_i)$ and $(D_i,g_i)$ in $\cC,$  we say that they are  \emph{isomorphic}  if  there is  a commutative diagram  in $\cC$:

\[
\begin{tikzcd}[font=\large] 
C_{m_1}  \arrow[swap]{d}{d_1} && C_{m_2}\arrow[swap]{d}{d_2}\arrow{ll}{f_{m_1,\,m_2}}&&C_{m_3}\arrow{ll}{f_{m_2,\,m_3}}\arrow[swap]{d}{d_3} &&  \cdots \arrow{ll}  \\ D_{n_1} && D_{n_2} \arrow{ll}{g_{n_1,\,n_2}}\arrow{llu}{u_1} &&D_{n_3} \arrow{ll}\arrow{ll}{g_{n_2,\,n_3}}\arrow{llu}{u_2}  &&\cdots \arrow{ll}   
\end{tikzcd}
\]
Such a diagram will be referred to as an \emph{isomorphism}. 

Equivalently, two inverse sequences $(C_i,f_i)$ and $(D_i,g_i)$ in $\cC$ are isomorphic if there is an inverse sequence in $\cC$ which can be truncated and telescoped to form both a subsequence of $(C_i,f_i)$ and a subsequence of $(D_i,g_i).$ For example, given the above isomorphism, such an interwoven sequence would be given by a diagram of the form 

\[
\begin{tikzcd}[font=\large] 
D_{n_1}  \arrow[swap] && C_{m_1}\arrow{ll}{d_1}&&D_{n_2}\arrow{ll}{u_1} && C_{m_2}\arrow{ll}{d_2}&&  \cdots \arrow{ll} \,.
\end{tikzcd}
\]
This latter approach is easily adapted to more general \emph{inverse systems} and is taken by Fox \cite{F1}. The work of this article could readily be written in such more general terms, but there are presentational and conceptual conveniences to remaining with inverse sequences. 

A \emph{morphism} of inverse sequences is defined similarly by a ladder diagram with vertical morphisms going in just one direction. 

We denote the category of inverse sequences in $\cC$ by $\overline{\cC}.$ Technically speaking, what we are referring to as an isomorphism in $\overline{\cC}$ is not an isomorphism of objects in $\overline{\cC}$ but rather in the category pro-$\cC$, a category whose objects are inverse systems and whose morphisms are equivalence classes of morphisms in the category of inverse systems. For example, a commutative diagram as above represents two morphisms whose compositions are \emph{not} the identity in $\overline{\cC}$, but in pro-$\cC$ this composition is equivalent to the identity. However, the existence of what we are referring to as an isomorphism establishes an equivalence relation on $\overline{\cC}$, and so the abuse of terminology is mild, and we spare the reader the details of formulating the  pro-$\cC$ category as the details are not relevant to our investigation.

\subsection{Inverse sequences and inverse limits.} We emphasise that a key feature of our invariants is that they occur on the level of  inverse sequences and not  their limits. While the limits (when they exist) of two isomorphic inverse sequences are isomorphic as objects in $\cC$, it is easy to find sequences which are not isomorphic but which have isomorphic limits. Passing to limits thus severely degrades the information contained  in the inverse sequence. For example,  consider the two sequences $C=(C_i,f_i)$ and $D=(D_i,g_i)$ in the category of groups, where $C_i$ is the trivial group for all $i$ and $D_i$ is $\mZ$ with $g_i(z)=2z$ for all $i$. While these two inverse sequences are not isomorphic (it is impossible to factor the bonding maps in $D$ through any of the $C_i$), they both have the trivial group as limit.

While our invariants are inverse sequences and not their limits, we will need to consider the limits of sequences of topological spaces, and for ease of notation we make use of a specific representation of this limit. Given a sequence of compact metric spaces $X_i$, $i\in \mZ^+$ and continuous maps $f_i\colon X_{i+1} \to X_i$, the \emph{inverse limit} of the sequence $X_1 \xleftarrow{f_1} X_2 \xleftarrow{f_2} X_3 \xleftarrow{f_3} \cdots$ is the subspace of the product space,
\[
\underleftarrow{\lim}(X_i,f_i)=\left\{ \,(x_i)_{i\in \mZ^+} \in \prod_{i\in \mZ^+}X_i \,: \text{ for all }i\in \mZ^+, \;x_i=f_i(x_{i+1}) \, \right\}.
\]
This is topologised with the subspace topology of the product topology on $\prod_{i\in \mZ^+}X_i $, but if  $X_i$ has the metric $d_i,$ we give the product $\prod_{i\in \mZ^+} X_i $ the explicit metric 

\[
d\left((x_i),(y_i) \right) = \sum_{i\in \mZ^+} \frac{d_i(x_i,y_i)}{2^i\left( 1+d_i(x_i,y_i)\right)}\,.
\]
Observe that if $p_n$ denotes the projection of  $\underleftarrow{\lim}(X_i,f_i)$ to $ X_n$, then for any $x\in X_n$ we have 
\[
\text{diam } p_n^{-1}(x) < \frac{1}{2^n}\, .
\]

\section{Structure Theorems}\label{Sectexpan}

Essential to our results is identifying and using  appropriate representations of our flow spaces. This requires us to define particular types of presentations of flow spaces as limits of inverse sequences of wedges of circles, which we term \emph{expansions} of the spaces. This type of representation of a flow space leads to a direct link with a symbolic representation of it. We begin with introducing the category in which the expansions take place.

\subsection{The wedge category $\cW$.}

\begin{defn}
For $n\in \mZ^+$, let $W_n$ denote the topological space formed by joining $n$ circles at a point. That is, if $S_1,\dots, S_n$ are disjoint circles with corresponding base points $s_i$, then $W_n$ is the quotient space $\cup_i\, S_i/\!\sim,$ where $s_1\sim s_2 \sim \cdots \sim s_n.$ We refer to $W_n$ as the \emph{wedge of $n$ circles} and $w_n$, the equivalence class of  the $s_i$, as the \emph{wedge point}, or the \emph{base point}.
\end{defn}

Notice that $W_n$ is determined up to homeomorphism by this definition, but there are many geometrically different circles that could be used to form the same $W_n.$

\begin{defn}
An \emph{orientation} on $W_n$ is a choice of orientation for each of the circles used to form $W_n$. A \emph{positive map} between oriented wedges of circles $f\colon W_m \to W_n$ is a continuous, surjective map which maps wedge point to wedge point, preserves the orientation of each circle and is such that each point of $W_m$ apart from the wedge point has a neighbourhood which is mapped injectively by $f.$
\end{defn}

\begin{defn}
$\cW$ is the \emph{wedge category}, whose objects are oriented wedges of circles $W_n$ and whose morphisms are positive maps of wedges of circles.  
\end{defn}

Observe that the homotopy class of a positive map  is completely determined by specifying the circles in the codomain (together with their orientations) to which the circles  in the domain space are mapped. This leads naturally to the following notion.

\begin{defn}
Given a positive map $f: (X_2,x_2) \to (X_1,x_1)$ from an oriented wedge of $\ell$ circles $X_2$ to an oriented wedge of $k$ circles $X_1$, we form a \emph{symbolic representation} as follows. 

First assign to each circle $X^1_i$ in $X_1$ a symbol $a_i$ and assign to each circle $X^2_j$ in $X_2$ a symbol $b_j.$ Then if $f$ maps $X^2_i$ to the circles $X^1_{j_1},\dots,X^1_{j_m}$, in that order, we define $f_s(b_i)=a_{j_1}\cdots a_{j_m}$. The function $f_s\colon \{b_1,\dots,b_\ell\}\to \{a_1,\dots,a_k\}^*$ thus determined from the set of symbols $\{b_1,\dots,b_\ell\}$ to the set of finite strings (words) on the symbols $\{a_1,\dots,a_k\}$ is then a \emph{symbolic representation}. 

When necessary, we refer to the set of symbols $\{a_1,\dots,a_k\}$ used to represent symbollically the circles in $X_1$ as its \emph{alphabet}, denoting it with notation such as $\cA_1$.
\end{defn}

Thus, to such a positive map is associated a substitution in the sense of symbolic dynamics, except that the alphabets might vary. What are termed \emph{proper} substitutions, whose definition is similar to that below, play a special role in the study of  substitutions, see, for example, \cite{DHS},\cite{BD}. It is perhaps then not surprising that proper positive maps play an important role in our study of representations of flow spaces.

\begin{defn}
A positive map $f: (X_2,x_2) \to (X_1,x_1)$ of oriented wedges of circles is \emph{proper} if in any symbolic representation $f_s\colon \mathcal{A}_2 \to \mathcal{A}^*_1$  for every  $a_i\in \mathcal{A}_2$, $f_s(a_i)$  starts with the same symbol and ends with the same symbol and hence contains at least $2$ symbols. That is, there exist some $c,d \in \mathcal{A}_1$ (which could coincide), such that for all $a_i\in \mathcal{A}_2,$ $f_s(a_i)= c v_id$ for some word $v_i\in\cA_1^*$. The subword $v_i$ between $c$ and $d$ may vary with $i$, and may be empty. Similarly, we define a map of finite alphabets $\cA\to\cB^*$ as \emph{proper} if it represents a proper positive map of wedges of circles.
\end{defn}

Observe that since positive maps are surjective, a proper map of finite alphabets $s\colon \{a_1,\dots,a_k\}\to \{b_1,\dots,b_\ell\} ^*$ satisfies the property that for each $b_j$ there is some $a_i$ with $b_j$ appearing in $s(a_i)$ . 

While proper maps are crucial to the structure of aperiodic flow spaces, the following notion is key for the structure of more general flow spaces.

\begin{defn}
A positive map $f: (X_2,x_2) \to (X_1,x_1)$ of oriented wedges of circles is \emph{degenerate proper} if in any symbolic representation $f_s: \mathcal{A}_2 \to \mathcal{A}^*_1$ for every symbol $b\in \mathcal{A}_2$, $f_s(b)$ starts and ends with the same single symbol $a$, and for at least one $b\in \mathcal{A}_2$, $f_s(b)=a.$  Similarly, we define such a map of finite alphabets as \emph{degenerate proper} if it represents a degenerate proper positive map of wedges of circles.
\end{defn}

Observe that the composition of any number of positive maps followed by a proper map is proper and words in the values of the symbolic map representing  the composition of $n$ proper maps have at least $2^n$ symbols. In contrast,  the composition of any number of positive maps followed by a degenerate proper map can either proper or degenerate proper.

\begin{defn}
A \emph{projection} from a flow space $\mX$ to an oriented wedge of circles $W_n$ is a surjective continuous map $p\colon \mX \to W_n$ satisfying the following conditions :
\begin{enumerate}
    \item For each point $x\in \mX$ there is a corresponding $\e_x>0$ so that the restriction $p|(-\e_x,\e_x).\,x$ is injective. That is, $p$ is locally injective along orbits.
    \item Endowing the orbits of $\mX$ with the orientation provided by the flow, $p$ preserves the orientation of each flow orbit. 
\end{enumerate}

\end{defn}

While projections do not generally preserve the action of the flow (merely their direction), they do have some key properties. 

\begin{lemma}\label{projectionfacts}
Suppose  $p\colon \mX \to W_n$ is a projection. For any point $x\in \mX$, there is an unbounded sequence of $t>0$ for which $p(t.x)=w_n$. And for any $w\in W_n,$ $p^{-1}(w)$ is a totally disconnected subspace of $\mX.$
\end{lemma}

\begin{proof} 

Suppose  that there were a  $T>0$ so that for any $t>T$ we have $p(t.\,x)\neq w_n$. Then as $\mX$ is compact, the closure of $(T,\infty).\,x$ contains a minimal set $M$, which is not a fixed point by our hypothesis. But $p\left( (T,\infty).\,x\right) $ must be contained in a single circle $X_i^1$ as otherwise this image would contain $w_n.$ As $p$ preserves orienatation and is locally injective along $(T,\infty).\,x$, $p(t.\,x)$ must limit to a single point $w$ (possibly $w_n$) as $t \to \infty$. But then by continuity we would have $p(M)=w$, contradicting the local injectivity of $p$ along the orbits of points in $M$. Thus, there must be an unbounded sequence of $t>0$ for which $p(t.x)=w_n$. 

Let  $w\in W_n\setminus\{w_n\}$ and consider any point $x \in p^{-1}(w)=F.$ Then for a zero--dimensional compact set $Z$ there is a matchbox neighbourhood $U$ of $x$ homeomorphic to $Z\times (-1,1)$ contained in the open set $\mX\setminus p^{-1}(w_n).$ Then, for each $z\in Z$ there is at most one point of $F$ in the subset of $\mX$ in the match corresponding to $\{z\}\times (-1,1)$ as otherwise there would be a point of $U$ whose orbit would have $p$ image intersecting $w_n$. Thus, $F \cap U $ can be identified with a subset of $Z$, and so there is a local basis for $F$ at $x$ consisting of clopen sets. As this is true for all points of $F,$ $F$ is totally disconnected. A similar argument applies to $p^{-1}(w_n),$ only we use a set of $n$ points consisting of one point (other than $w_n$) from each circle in $W_n$ to play the role that $w_n$ played in the previous case.  

\end{proof}

\subsection{Expansions and flow expansions.}
\begin{defn}\label{expansion}
For a flow space $\mX$ and point $x_\bullet \in \mX$, an \emph{expansion of $\mX$ about $x_\bullet$} consists of an inverse sequence of oriented wedges of $k_n$ circles $X_n$ with wedge point $x_n$
\[
X_1 \xleftarrow{f_1} X_2 \xleftarrow{f_2} X_3 \xleftarrow{f_3} \cdots
\]
satisfying the following conditions:
\begin{enumerate}
    \item Each bonding map $f_n$ is a positive map.
    \item For each $n\in \mZ^+$ there is a projection $p_n: (\mX, x_\bullet) \to (X_n,x_n)$ so that:
\begin{enumerate}
\item For $k\in \mZ^+$, $p_n = f_{n\, , \, n+k} \circ p_{n+k}$.
\item The map $h\colon \mX \to  \underleftarrow{\lim}\{X_1 \xleftarrow{f_1} X_2 \xleftarrow{f_2} X_3 \xleftarrow{f_3} \cdots\}$ defined by $h(x)=(p_n(x))_{n\in \mZ^+}$ is a homeomorphism.

\end{enumerate}   
\end{enumerate}

\end{defn}
We denote an expansion about $x_\bullet$ by $(\mX,x_\bullet)\approx\eX.$ Observe that the limits are in the category of compact metric spaces and not in $\cW$.

It is not always the case that a flow space contains a point about which there is an expansion, and there are some flow spaces which admit an expansion about some of its points but not about others. The following notion is the key to distinguishing the cases, and in particular will give us in Theorem \ref{expansionexists} a necessary and sufficient for a flow space to have an expansion.

\begin{defn}
The flow space $\mX$ is \emph{dense at} the point $x\in \mX$ if every neighbourhood of $x$ contains a point from the forward flow orbit of each point of $\mX$. 
\end{defn}

\begin{prop}
If the flow space $\mX$ is not dense at $x_\bullet\in \mX,$ then $\mX$ does not admit an expansion about $x_\bullet.$
\end{prop}

\begin{proof}
Suppose $\mX$ is not dense at $x_\bullet$ but that  we have an expansion about $x_\bullet$ as in Definition \ref{expansion}. Then there is some neighbourhood $U$ of $x_\bullet$ that does not intersect the forward orbit of some $y\in \mX.$ Let $\e>0$ be less
than the distance from $x_\bullet$ to the complement of $U.$ As each $p_n$ is a projection, by Lemma \ref{projectionfacts} for each $n\in \mZ^+$ there is a $y_n$ on the forward orbit of $y$ with $p_n(y_n)=x_n=p_n(x_\bullet).$ By construction, $d(y_n,x_\bullet)>\e$. Since $h$ as in the Definition \ref{expansion} is a homeomorphism, the fibres of the maps $p_n$ must have diameters converging to $0$ as $n \to \infty.$ But we have that the $p_n$ fibre of each $x_n$ is greater than $\e$ in diameter, a contradiction. 

\end{proof}

\begin{ex}
Consider the subshift $S$ of $\{0,1\}^{\mZ}$  generated by the points $(\mathbf{0})$, the point consisting of all $0$'s, and $(\dots,0,0,1,0,0,\dots)$, the point which is $0$ except at $n=0$. The suspension of $S$ will be dense at any point on the suspension orbit of $(\mathbf{0})$ but not dense at any point of the other orbit. 
\end{ex}

A single minimal set together with orbit(s) limiting to it will give similar examples, but any flow space with two distinct minimal sets will not be dense about any point. 

\begin{lemma}\label{periodic}
Suppose $\mX$ admits an expansion  $(\mX,x_\bullet)\approx\eX$. If there is an $n\in \mathbb{Z}^+$ so that for all $m>n$, $f_{m,n}$ is degenerate proper, then $\mX$ contains a periodic orbit.
\end{lemma}
\begin{proof}
Suppose for a given expansion of $\mX$, $n$ is such that for all $m>n$, $f_{m,n}$ is degenerate proper. This means that there is a circle, $X^n_1$ in $X_n$ say, so that for each $m>n$ we have a circle, $X^m_1$ in $X_m$ say, such that $f_{m,n}\left( X^m_1\right) = X^n_1$ and the $f_{m,n}$ image of $X^m_1$ only covers $X^n_1$ once. Thus, the restriction of $f_{m,n}$ to $X^m_1$ is a homeomorphism onto $X^n_1$. Consider the truncation of the inverse sequence formed by starting the sequence with $n,$ whose limit is still homeomorphic to $\mX.$ Then the map $e$ from $X^n_1$ into the limit of the truncated sequence formed by sending a point in $X^n_1$ to itself in the $n^{\text{th}}$ co-ordinate and for $m>n$ to its image under the inverse of the restriction of  $f_{m,n}$ to $X^m_1$. Then $e$ is an embedding of $X^n_1$ into the limit space, and any circle in a flow space without fixed points is a periodic orbit.
\end{proof}
As a flow space $\mX$ admitting an expansion has only one minimal set, any periodic orbit in $\mX$ would be its unique minimal set. Recall that we use \emph{aperiodic} to refer to a flow space $\mX$ with no periodic orbit.

\bigskip
As for projections, an expansion of a flow space $\mX$ preserves only the direction of the flow, not the flow action itself. However, it is convenient also to introduce a special class of expansions that do respect the flow action. Of course, for $n>1$, $W_n$ does not support a flow due to the wedge point, and this leads to the following notions.

\begin{defn}
A \emph{partial flow} on $W_n$ consists of a non-constant, periodic flow on each of the circles $S_i$  used to form $W_n.$    
\end{defn}

In the special case of $n=1,$ a partial flow on $W_n$ is a flow. We consider $W_n$ endowed with a partial flow to be oriented by the positive direction of the partial flow.  Given a projection $p:\mX \to W_n$ where $W_n$ has a partial flow, we shall construct a map of flow spaces $\widehat{p}$ from $\mX$ to a tiling space $\cT(p)$ as explained below. 

\begin{defn}\label{global tiling space}
Suppose $W_n$ is endowed with a partial flow for which $S_i$ has period $P_i$. The \emph{global tiling space } $\cT(W_n)$  is the suspension of the shift  $\left( \{a_1,\dots,a_{n} \}^\mZ, s  \right)$  under the ceiling function which takes on the value $P_i$ over the cylinder set $ \{ (x_n)_{n\in \mZ}\colon x_0=a_i \}$.
\end{defn}
 
To a point $x\in \mX$ there is a naturally associated element $\tau (x) \in \{a_1,\dots,a_{n} \}^\mZ$ given by the sequence of $S_i$ that $p$ maps the orbit of $x$ to, with the understanding that if $p(x)=w_n$, the  $0^{\rm th}$ entry of $\tau (x)$, which we denote $(\tau (x))_0$\,, is $a_i$, where $S_i$ is the first circle the forward orbit of $x$ is mapped to by $p.$ Then define the continuous map  $\widehat{p}\colon \mX \to \cT(W_n)$ by $\widehat{p}(x)=(\tau(x), t)$, where $t$ satisfies $p(x)=t\,.\,w_n=t\,.\,s_i \in S_i \subset W_n,$ where $w_n$ is the  wedge point considered as the point $s_i\in S_i$ and $i$ corresponds to $(\tau (x))_0$.

\begin{defn} \label{tilingfactordef}
Suppose $p\colon\mX \to W_n$ is a projection and $W_n$ has a partial flow. With $\widehat{p}\colon \mX \to \cT(W_n)$ the induced map, denote the image of $\widehat{p}$ by $\cT(p)$.
\end{defn}

Given a projection $p:\mX \to W_n$, let $\bar{p}\colon \cT(W_n) \to W_n$ be the map that for $t\in [0,P_i]$ sends the set $\{(t,(x_n)_{n\in \mZ}) \colon x_0=a_i \}$ to $t\,.\,w_n=t\,.\,s_i \in S_i \subset W_n,$ where $w_n$ is the  wedge point considered as the point $s_i\in S_i.$  By construction we have $\bar{p}\circ \widehat{p}=p$.

In general, there is no reason to expect $\widehat{p}$  to be injective or to respect the flow actions.

\begin{defn}
A projection $p\colon \mX \to W_n$ from a flow space to $W_n$ endowed with a partial flow is a \emph{factor map} if $\widehat{p}\colon \mX \to \cT(p)$ is a factor map, i.e.,  $\widehat{p}(t\, .\, x) = t\,. \,\widehat{p}(x).$
\end{defn}

Thus,  outside the fiber of $w_n$, $p$ locally conjugates the flow on $\mX$ with the partial flow on $W_n.$ 

Given a partial flow on $W_m$ with period $P_i$ on the circle $S_i,$ let $\widetilde{W}_m$ be a periodic flow space of period $\sum_iP_i$, and let $p \colon \widetilde{W}_m \to W_m$ be a factor map which winds around the circles $S_i$ once in succession (the order of succession is not important).
\begin{defn}
If $W_m$ and  $W_n$ are endowed with partial flows and $p$ and $\widetilde{W}_m$ are as above, a positive map $f\colon W_m\to W_n$ is a \emph{factor map} if    $f \circ p \colon \widetilde{W}_m \to W_n$ is a factor map.
\end{defn}

Observe that in this case both $\cT(p)$ and $\cT(f\circ p)$ are periodic. 
 
\begin{defn}\label{flow expansiondef}
For a flow space $\mX$ and point $x_\bullet \in \mX$, a \emph{flow expansion of $\mX$ about $x_\bullet$} consists of an inverse sequence of wedges of $k_n$ circles $X_n$ endowed with a partial flow and with wedge point $x_n$ 
\[
X_1 \xleftarrow{f_1} X_2 \xleftarrow{f_2} X_3 \xleftarrow{f_3} \cdots
\]
satisfying the following conditions:
\begin{enumerate}
    \item Each bonding map $f_n$ is a factor map.
    \item For each $n\in \mZ^+$ there is a factor map $p_n: (\mX, x_\bullet) \to (X_n,x_n)$ so that:
\begin{enumerate}
\item For $k\in \mZ^+$, $p_n = f_{n\, , \, n+k} \circ p_{n+k}$.
\item The map $h\colon \mX \to  \underleftarrow{\lim}\{X_1 \xleftarrow{f_1} X_2 \xleftarrow{f_2} X_3 \xleftarrow{f_3} \cdots\}$ defined by $h(x)=(p_n(x))_{n\in \mZ^+}$ is a homeomorphism.

\end{enumerate}   
\end{enumerate}

\end{defn}
We denote such a flow expansion about $x_\bullet$ by $(\mX,x_\bullet)\approx_f\eX.$

Observe that this notion of flow expansion is not purely topological and includes the flow in an essential way and that a flow expansion is a special type of expansion.

\begin{theorem}\label{flow expansion}
If the flow space $\mX$ is dense at $x_\bullet$ and if there is a factor map $p_1:(\mX,x_\bullet) \to (X_1,x_1)$ to a wedge of circles with wedge point $x_1,$ then $\mX$ admits a flow expansion about $x_\bullet .$ If $\mX$ is aperiodic, then the flow expansion can be chosen so that each bonding map $f_n$ is proper; otherwise, $\mX$ has a periodic orbit as its unique minimal set and it admits a flow expansion about $x_\bullet$ for which $f_{n,1}$ is degenerate proper for each $n.$
\end{theorem}

\begin{proof}
Let $T_1= p_1^{-1}(x_1),$ and let $k_1$ denote the number of circles $X^1_i$ used to form $X_1$, with $P^1_i$ denoting the period of the flow on $X^1_i$.  We partition $T_1$ into $k_1$ disjoint clopen sets $K^1_1,\dots,K^1_{k_1}$ as follows. Letting $t_0 = \frac{\min \{P^1_i\}}{2}$, we have that $p_1\left( t_0 \, . T_1 \right)$ is partitioned into $k_1$ disjoint closed (and hence clopen) sets  $J^1_1,\dots,J^1_{k_1},$ where $J^1_i = p_1\left( t_0 \, . T_1 \right) \cap X^1_i$. Setting $K^1_i = -t_0 \, . p_1^{-1} \left(J^1_i\right)$ gives us the desired partition of $T_1.$  Thus, for any two points of $y,z\in K^1_i,$ we have that $p_1(t.y)=p_1(t.z)$ in $X^1_i$ for $t\in [0,P^1_i]$, and $\mX$ is the union of the ``tubes" $[0,P^1_i].K^1_i$, which intersect only in $T_1$. Let $k>0$ be the minimum distance between any two $K^1_i.$ 

Similar to the $K^1_i,$ $T_1$ is partitioned into $k_1$ clopen sets $L^1_1,\dots,L^1_{k_1},$ where $L^1_i$ consists of those points of $T_1$ whose orbit is mapped by $p_1$ onto the circle $X^1_i$ immediately before hitting $T_1$ in the flow.  Without loss of generality, assume $x_\bullet\in K^1_1 \cap L^1_1.$ Now choose a clopen neighbourhood $T_2\subset K^1_1 \cap L^1_1$ of $x_\bullet $ in $T_1$ with diameter less than $\frac{1}{2}.$ As the flow is dense at $\mX$, we know that the forward orbit of each point in $T_2$ will return to $T_2$, and by the compactness of $T_2,$ there is a maximum such return time $R.$ As $T_2$ is clopen, the return time to $T_2$ will be locally constant. By the uniform continuity of the flow action on $[0,R].T_1,$ there is a $\delta>0$ so that if two points of $x,y\in T_1$ satisfy $d(x,y)<\delta,$  then  $d(t.x,t.y)<\min (\frac{1}{2},k)$ for all $0\leq t \leq R$. (Note, the $\frac{1}{2}$ is to provide a uniform control on the diameter of fibres of the projection and is adjusted to $\frac{1}{n}$ in stage $n$ of the construction.)

Now partition $T_2$ into finitely many, say $k_2$, clopen sets $K^2_1,\dots,K^2_{k_2}$ of diameter less than $\delta$ and so that each point of $K^2_i$ has the same return time $P^2_i$ to $T_2.$ By construction, $\mX$ is the union of $k_2$ tubes $ [0,P^2_i].K^2_i$ which intersect only in $T_2$, and, by the choice of $\delta$ and $k$, each individual tube $ [0,P^2_i].K^2_i$ is contained in a single $[0,P^1_j].K^1_j$ tube between successive (with respect to the flow) intersections with $T_1$.

Let $(X_2,x_2)$ be a wedge of $k_2$ circles $X^2_i$ endowed with a partial flow with period $P^2_i$ on $X^2_i$. Define $p_2:(\mX,x_\bullet)\to (X_2,x_2)$ by $p_2(T_2)=x_2$, and for any point in $y\in t.K_i^2$  $(t\in [0,P_i^2])$ by $p_2(y)=t.x_2$ in $X^2_i$, where by abuse of notation $t.x_2$ denotes the point of $X^2_i$ given by the action of $t$ in the partial flow action on the point of $X^2_i$ that is identified to $x_2.$ By construction, $p_2$ is a factor map. There is then the well--defined factor map $f_1:(X_2,x_2)\to (X_1,x_1)$ which maps  $x_2$ to $x_1$ and generally maps a point in   $p_2\left( [0,P^2_i].K^2_i \setminus T_1\right)$ to the point in $X_1$ determined by the condition $p_1=f_1\circ p_2$. This is well--defined by the condition that each individual tube $ [0,P^2_i].K^2_i$ is contained in a single $[0,P^1_j].K^1_j$ tube between successive intersections with $T_1$. Note that as  $T_2\subset K^1_1 \cap L^1_1$, $f_1$ is a proper or degenerate proper factor map. 

Assume now that we have defined the projections $p_i: \mX \to X_i$ for $i\leq n$ and corresponding proper or degenerate proper factor maps $f_i$, $i<n$ for a given $n\in \mZ^+$ satisfying the conditions of Definition \ref{flow expansiondef} and with the diameter of each set $T_i=p_1^{-1}(x_i)<1/i$ and the fibers of each $p_i$ less than $\frac{1}{i}$ in diameter. Repeat the above construction with $T_n$ playing the role of $T_1$ and replacing $\frac{1}{2}$ with $1/(n+1)$ as appropriate to define the factor map $p_{n+1}: \mX \to X_{n+1}$ and proper  or degenerate proper bonding map $f_n: X_{n+1}\to X_n$. Thus, we have recursively defined a sequence of factor maps and proper or degenerate proper bonding maps for all $n\in \mZ^+$ in such a way that there is a well-defined map $h: \mX \to  \underleftarrow{\lim}(X_i,f_i)$ given by $h(x)=(p_n(x))_{n\in \mZ^+}.$ By construction, $h$ is a continuous surjection. That $h$ is injective follows from the fact that the diameters of the fibres of the maps $p_n$ converge to $0$ as $n\to \infty.$

We then have two cases. In the first case, for each $n\in \mathbb{Z}^+$ there is a corresponding $m>n$ so that the composition $f_{m,n}$ is proper. In this case, by telescoping appropriately, we see that $\mX$ admits a flow expansion in which each bonding map is proper. Moreover, in this case the periods of circles in the factor spaces grow without bound and so $\mX$ is aperiodic.

In the second case, there is an $n\in \mathbb{Z}^+$ so that for all $m>n$, $f_{m,n}$ is degenerate proper. We know that in this case $\mX$ contains a periodic orbit by Lemma \ref{periodic} and the result follows after truncation and re-indexing.

\end{proof}

We could have constructed a flow expansion without requiring $T_2\subset K^1_1 \cap L^1_1$ and analogously throughout the construction, but there would be no guarantee that the resulting flow expansion would have bonding maps which are (degenerate) proper. However, as $\text{diam } p_n^{-1}(x_n)= \text{diam } T_n \to 0$ as $n\to \infty$ in any expansion, we must have for sufficiently large $n$ that $T_n\subset K^1_1 \cap L^1_1$ and similarly for each $K^i_1\cap L^i_1$. Thus for an aperiodic flow space, by telescoping appropriately, the sequence will consist of proper bonding maps even if the original one does not. Note that this remark applies to both flow expansions and expansions as it merely relates to the shrinking diameters of the $p_n$ fibres of the base points. 

We are now able to establish the existence of a (general, i.e., not necessarily flow) expansion of a flow space $\mX$ whenever there is a dense point.

\begin{theorem}\label{expansionexists}
Suppose the flow space $\mX$ is dense at the point $x_\bullet.$  If $\mX$ is aperiodic, then $\mX$ admits an expansion about $x_\bullet$ for which every bonding map is proper. Otherwise, $\mX$ admits an expansion about $x_\bullet$ for which $f_{n,1}$ is degenerate proper for each $n.$
\end{theorem}
\begin{proof}
We know by \cite{KS},\cite{AM} that there is a flow equivalence $h$ from $\mX$ to the suspension $S$ of a homeomorphism of a totally disconnected set. The property of being dense at a point is preserved by flow equivalence, and so $S$ is dense at $h(x_\bullet)$. As any suspension flow admits a genuine factor map onto the circle, the result follows directly from Theorem \ref{flow expansion} as via $h$ we obtain an expansion of $\mX$ from any flow expansion of $S.$

\end{proof}

Note that in many cases, such as when $\mX$ is topologically weakly mixing, there is no such $h$ that preserves the flow action, and so there is no reason to assume that $\mX$ itself admits a flow expansion about $x_\bullet$. However, when $\mX$ is minimal, it is dense at each point and thus we have

\begin{cor}
If $\mX$ is a minimal flow space, then it admits an expansion about every point.
\end{cor}

\begin{theorem}\label{gen exp to flow exp}
If $\mX$ is a flow space with expansion $(\mX,x_\bullet)\approx \eX $, then by altering the flow on $\mX$ and endowing the $X_i$ with appropriate partial flows, we can obtain a flow expansion $(\mX,x_\bullet)\approx_f\eX$ using the same bonding maps and projections.
\end{theorem}

\begin{proof}
Suppose $\mX$ is a flow space with expansion $(\mX,x_\bullet)\approx\eX$. We judiciously construct a homeomorphism from $\mX$ to a suspension and pull back the suspension flow to obtain the result.

Consider then the projection $p_1\colon \mX \to X_1$, where $X_1$ is formed from the $k_1$ circles $X^1_i$. Now parameterise each $X^1_i$ via an orientation preserving covering map $\rho_i \colon \mR \to X^1_i$ for which $\rho_i ( \mZ) = x_1$ for each $i.$ For each $t\in [0,1),$ let $F_t = \cup_i \, p_1^{-1}\left(\rho_i(t)\right),$ where $F_0$ is the same as $p_1^{-1}(x_1).$ By Lemma \ref{projectionfacts} every flow orbit meets $F_0$ and $F_0$ is totally disconnected.  Let $r\colon F_0 \to F_0$ be the return homeomorphism of the flow to $F_0.$ For a given $x\in F_t,$ let $\cR (x)$ be the first return of $x$ to $F_0$ under the reverse flow: $\cR\colon \cup_{t\in [0,1)} F_t =\mX \to F_0$, is given by $\cR(x) = -t_x.\,x,$ where $t_x = \min \{t\geq 0 \colon -t.\,x \in F_0\}. $ For each $t\in [0,1)$, the restriction $\cR|F_t\to F_0$  is a homeomorphism and is the identity for $t=0$. Consider the suspension $S$ of $r$,  which we regard as the quotient of $F_0 \times [0,1].$ Now we define the map $h \colon \mX \to S$, where for $x\in F_t$, $h$ send $x \mapsto (\cR(x), t) \in F_0 \times [0,1)$. Observe that each $x$ is in precisely one $F_t$ and  $h$ maps $F_t$ homeomorphically onto $F_0 \times \{t\}$. Moreover, the identification of $F_0 \times \{t\}$ in $S$ matches identification of what would be $F_1$ in $\mX$ with $F_0$ and so $h$ is a homeomorphism. Now, endow $\mX$ with the flow $\phi$ given by the pullback via $h$ of the suspension flow on $S.$ By construction, the projection $p_1$ is then a factor map,  where $X_1$ is given the partial flow for which each $X^1_i$ is given the flow of period $1$ induced by $\rho_i$. Then, we pull back the flow on $X_1 $ via $f_{n, }$ to obtain partial flows on $X_n$. By the commutativity of the original maps, the projections $p_n$ then become factor maps, as are the bonding maps. Thus, relative to $\phi$, we now have  $(\mX,x_\bullet)\approx_f\eX $ as desired.

\end{proof}

\subsection{Symbolic sequences which give flow spaces.}

We now determine which inverse sequences in the wedge category $\cW$ correspond to a flow space, and indeed which of these are minimal. It will be convenient for subsequent investigation to do so on the symbolic level, which will require us to introduce some further definitions.

\begin{defn}
Given finite alphabets $\mathcal{A}_i$ ($i=1,2,3$) and maps $s_i: \mathcal{A}_{i+1}\to \mathcal{A}_i^* $ 
(where $\mathcal{A}_i^*$ is the set of finite words in symbols from $\mathcal{A}_i$), the \emph{composition} $s_i\circ s_{i+1}$ is the map $\mathcal{A}_{i+2}\to \mathcal{A}_i^* $ assigning to $a\in \mathcal{A}_{3} $ the concatenated word $s_i(b_1)\cdots s_i(b_k)$, where $s_{i+1}(a)=b_1\cdots b_k$.
\end{defn}
The composition of multiple maps is defined similarly using further concatenation.

\begin{defn}
We say that a sequence $\cS$ of finite alphabets $\mathcal{A}_i$ and maps $s_i: \mathcal{A}_{i+1}\to \mathcal{A}_i^* $ for $i\in \mathbb{Z}^+,$ \emph{represents a flow space} if there is a flow space $\mX$ with an expansion $(\mX,x_\bullet)\approx\eX $ for which  $s_i$ is a symbolic representation of  $f_i$ for each $i$.
\end{defn}
Observe that if $\cS$ represents a flow space even with a flow expansion, the flow space is not uniquely determined up to conjugacy  as we can freely choose the periods in say $X_1$, and this can change the conjugacy class, even after rescaling \cite{CS}.  However, as we now show, the flow equivalence class and the isomorphism class of the inverse sequence in $\overline{\cW}$ of any expansion is the same for any flow space $\cS$ represents.

\begin{lemma}\label{seq rep equiv}
If the symbolic sequence $\cS$ represents the flow spaces $\mX$ and $\mY$ with corresponding expansions $(\mX,x_\bullet)\approx\eX$ and $(\mY, y_\bullet )\approx\eY $, then there is a flow equivalence $h:(\mX,x_\bullet)\to (\mY,y_\bullet)$ and an isomorphism of the inverse sequences $(X_i,f_i)$ and $(Y_i,g_i)$ in $\overline{\cW}$.
\end{lemma}

\begin{proof}

Suppose $\cS$ represents the flow spaces $\mX$ and $\mY$ with corresponding expansions $(\mX,x_\bullet)\approx\eX $ and $(\mY, y_\bullet )\approx\eY .$ Then let $h_1:X_1 \to Y_1$ be a homeomorphism that is a positive map in $\cW$ which matches circles corresponding to the same symbol. As $f_1$ and $g_1$ have the same symbolic representation, we can see that there is a uniquely determined positive homeomorphism $h_2: X_2 \to Y_2$ satisfying $h_1 \circ f_1 = g_1 \circ h_2,$ and similarly with homeomorphisms $h_i$ making the following diagramm commutative: 

\[
\begin{tikzcd}[font=\large] 
X_{1}  \arrow[swap]{d}{h_1} && X_{2}\arrow[swap]{d}{h_2}\arrow{ll}{f_{1}}&&X_{3}\arrow{ll}{f_{2}}\arrow[swap]{d}{h_3} &&  \cdots \arrow{ll}  \\ Y_{1} && Y_{2} \arrow{ll}{g_{1}}&&Y_{3} \arrow{ll}\arrow{ll}{g_{2}} &&\cdots \arrow{ll}   
\end{tikzcd}
\]
This clearly gives a flow equivalence $(\mX,x_\bullet)\to (\mY,y_\bullet)$ and yields an isomorphism of $(X_i,f_i)$ and $(Y_i,g_i)$ in $\overline{\cW}$, but these are not canonical as there is choice in $h_1$.

\end{proof}
To characterise those symbolic sequences that represent a flow space, we require the following definition.

\begin{defn}
A sequence $\cS$ of finite alphabets $\mathcal{A}_i$ and maps $s_i: \mathcal{A}_{i+1}\to \mathcal{A}_i^* $ for $i\in \mathbb{Z}^+,$ is \emph{proper} if for each $n\in \mathbb{Z}^+$ there is an $m>n$ so that the composition $s_n \circ \cdots \circ s_{m-1}$ is proper.
\end{defn}

\begin{theorem}\label{symbolicexpansion}
A sequence $\cS$ of finite alphabets $\mathcal{A}_i$ and maps $s_i: \mathcal{A}_{i+1}\to \mathcal{A}_i^* $ for $i\in \mathbb{Z}^+$ represents an aperiodic flow space if and only if  $\cS$ is proper.
\end{theorem}
\begin{proof}
By the remarks following the proof of Theorem \ref{flow expansion}, we know that $\cS$ can represent an aperiodic flow space only if for each $n\in \mathbb{Z}^+$ there is an $m>n$ so that the composition $s_n \circ \cdots \circ s_{m-1}$ is proper. 

Suppose then that we have a proper sequence $\cS$ with $k_i$ denoting the number of symbols in $\mathcal{A}_i$.  By telescoping if necessary, we assume that each $s_i$ is proper for ease of notation. For each $i$, we let $X_i$ be a wedge of $k_i$ circles. We endow $X_1$ with the partial flow for which each period $P^1_j$ is $1.$ Then endow $X_2$ with the uniquely determined partial flow so that there is a continuous factor map $f_1\colon (X_2,x_2) \to (X_1,x_1)$ for which $s_1$ is a symbolic representation. Assuming we have endowed such a flow on $X_i$, recursively, endow $X_{i+1}$ with the uniquely determined partial flow so that there is a continuous factor map $f_i\colon (X_{i+1},x_{i+1}) \to (X_i,x_{i})$ for which $s_i$ is a symbolic representation. We now have an inverse system  and corresponding limit space $\mX=\underleftarrow{\lim}(X_i,f_i)$ with base point $x_\bullet = (x_n)_{n\in \mZ^+}$. As a first step to show that $\mX$ is a flow space, we show that it is a matchbox manifold.

Let $i\in \mathbb{Z}^+$ and let $z$ be any point of $X_i$. We shall show that $p_i^{-1}(z)$ is a totally disconnected subspace of $\mX$. For any $n>i$, $f_{i,n}$ is a factor map and so $f^{-1}_{i,n}(z)$ is a finite set, say $\{z_1,\dots, z_\ell\}$. Thus $\{p_n^{-1}(z_1),\dots p_n^{-1}(z_\ell)\}$ partitions $p_i^{-1}(z)$ into finitely many closed and hence clopen subsets. As each of these clopen sets has diameter less than $\frac{1}{2^n}$, we see that $p_i^{-1}(z)$ has a basis of clopen sets and is thus totally disconnected. 

We now show that each point of $\mX$ has a neighbourhood that is homeomorphic to $I\times Z$ for some closed interval $I\subset \mathbb{R}$ and some totally disconnected space $Z$. First consider a $x\in \mX$ that is not the base point $x_\bullet$. Then for some $i$, $z=p_i(x)\neq x_i.$ By the above, the space $Z:=p_i^{-1}(z)$ is totally disconnected. As $z\neq x_i$, we may choose a closed interval $I\subset \mathbb{R}$ containing $0$ in its interior such that $I.z$ does not contain the wedge point $x_i$ and is thus completely contained within a single circle used to form $X_i$. Thus, for each $n>i$, $f^{-1}_{n,i}\left(I.z\right)$ is partitioned by the intervals $\{I.z_1,\dots,I.z_\ell\}$ that do not contain the base point $x_n,$ where $f^{-1}_{i,n}(z)=\{z_1,\dots, z_\ell\}$. Thus, for  any $n>i$ and any any $w_n\in f^{-1}_{i,n}(z), $ the function $I\to X_n$ given by $t\mapsto  t.w_n$ is well--defined, injective and continuous.  By construction, for $w\in Z$ and for any $n>i,$ $p_n(w)\in f^{-1}_{i,n}(z).$ Consider then the map
$h: I\times Z \to p_i^{-1}\left( I.z \right) \hookrightarrow \mX$ given by  $h(t,w)=(r_n)_{n\in \mathbb{Z}}$, where for $n\geq i$, $r_n = t.\,p_n(w)$ and for $n<i,$
$r_n=f_{n,i}(t.z)$. Then $h$ is an injective and continuous map from a compact space and is thus a homeomorphism onto its image $p_i^{-1}\left( I.z \right)$. Hence, $ p_i^{-1}\left( I.z \right)$ is a neighbourhood of $x$ which is homeomorphic to $I\times Z$ as desired. 

Now consider the base point $x_\bullet$. As each $x_n$ is the base point of $X_n$, it typically belongs to more than one circle of $X_n$, and so $t.x_n$ is not generally defined as an element in $X_n$. As $P^1_j =1$ for each $j$ by construction, within each circle $X^n_i$, $(-\frac{1}{2},\frac{1}{2}).\,x_n$ does not completely cover $X^n_i$. Let $S_n$ be the ``spider" neighbourhood in $X_n$ of the base point $x_n$ formed by taking the union of intervals $(-\frac{1}{2},\frac{1}{2}).\,x_n$ in each circle used to form $X_n$. Similarly, for $t\in (-\frac{1}{2},\frac{1}{2})$, $t.\,x_n$ denotes the union of all such points in the individual circles.  As each $f_j$ is proper, for any $i<n$ the image $f_{i,n}(S_n)$ is an interval in $X_i$ and the map $(-\frac{1}{2},\frac{1}{2}) \to X_i$ given by $t\mapsto f_{i,n}(t.\, xX_n)$ is well-defined, continuous and injective. Letting $Z:= p_1^{-1}(x_1)$, consider the function $h: [-\frac{1}{4}, \frac{1}{4}]\times Z \to p_1^{-1}\left( [-\frac{1}{4}, \frac{1}{4}]. \, x_1 \right) \hookrightarrow \mX$ defined according to $2$ cases. Firstly, if $x_\bullet \neq w \in Z$ with say $p_j(w)\neq x_j$ then $h(t,w)=(r_n)_{n\in \mathbb{Z}}$, where for $n\geq j$, $r_n = t.\,p_n(w)$ and for $n<j,$ 
$r_n=f_{n,j}(t.\,p_j(w))$. Secondly, $h\left( t, x_\bullet \right)=(r_n)_{n\in \mathbb{Z}}$, where $r_n = f_n(t.\,x_{n+1})$. Then as before,  $h$ is an injective and continuous map from a compact space and is thus a homeomorphism onto its image $p_i^{-1}\left( I.z \right)$. Hence, $ p_i^{-1}\left( I.z \right)$ is a neighbourhood of $x$ which is homeomorphic to $I\times Z$ as desired.

Now that we have established that $\mX$ is a matchbox manifold, by the results of \cite{AHO}, $\mX$ is a flow space if it is orientable. This is, however, immediate as there is an orientation provided by the partial flow in each of the matchbox neighbourhoods considered above, and these orientations match on the overlap as they are both determined by partial flows related by a factor map. In fact, a flow on $\mX$ can be obtained from extending the given parameterisations of the neighbourhoods. 

\end{proof}
Observe that even without the assumption of properness, all points other than the base point are contained in a matchbox neighbourhood. If the properness condition is not met, then the path component of the base point will not be the continuous image of $\mathbb{R}$ and will have a number of rays emanating from the base point determined by the sets $\cap_{i>n}f_{n,i}(S_i),$ $n\in \mathbb{Z}^+$.

We know that any flow space admitting an expansion contains a single minimal set but may not be itself minimal. We now provide a criterion for minimality that is inspired by the analogous notion for substitutions, which has already been generalised quite broadly in \cite{FS} to other settings.

\begin{defn}
A positive map $f: (X_2,x_2) \to (X_1,x_1)$ of wedges of circles is \emph{primitive} if in any symbolic representation $f_s: \mathcal{A}_2 \to \mathcal{A}_1^*$ for every symbol $a\in \mathcal{A}_2$, $f_s(a)$ contains all the symbols of $\mathcal{A}_1.$ Similarly, we define a symbolic map satisfying this property as \emph{primitive}.
\end{defn}

\begin{defn}
A sequence $\mathcal{S}$ of finite alphabets $\mathcal{A}_i$ and maps $s_i: \mathcal{A}_{i+1}\to \mathcal{A}_i^* $ for $i\in \mathbb{Z}^+,$ is \emph{primitive} if and only if for each $n\in \mathbb{Z}^+$ there is an $m>n$ so that the composition $s_n \circ \cdots \circ s_{m-1}$ is primitive.     
\end{defn}

\begin{theorem}\label{minimal}
Let $\mX$ be an aperiodic flow space represented by the proper sequence $\mathcal{S}$ of finite alphabets $\mathcal{A}_i$ and maps $s_i: \mathcal{A}_{i+1}\to \mathcal{A}_i^* $ for $i\in \mathbb{Z}^+$ as in Theorem \ref{symbolicexpansion}. The flow on $\mX$ is minimal if and only if $\mathcal{S}$ is primitive. 
\end{theorem}

\begin{proof}
First observe that a flow is minimal if and only if every flow orbit is dense, and so the minimality of the flow on $\mX$ is independent of the flow on $\mX$ constructed from $\mathcal{S}$.  Without loss of generality by Theorem \ref{gen exp to flow exp}, we assume that the expansion corresponding to $\cS$ is a flow expansion. 

Now suppose that $\mathcal{S}$ is primitive and let $x$ be any point of $\mX.$ Recall that $\mX$ has a basis of open sets of the form  $p_n^{-1}(U),$ where $U$ is a non--empty basic open set in $X_n.$ Thus, to establish minimality we must show that the orbit of $x$ intersects any neighbourhood $p_n^{-1}(U)$ of such form. By the primitivity hypothesis, there is an $m>n$ so that $s_n \circ \cdots \circ s_{m-1}$ is primitive. As $p_m$ is a projection, the $p_m$ image of the orbit of $x$ will completely cover (at least) one of the circles $X^m_j$ used to form $X_m.$ By primitivity, we see that $f_{n,m}\left( X^m_j \right) $ contains all of $X_n$. Hence, the orbit of $x$ intersects  $p_n^{-1}(U)$ as required.

Suppose then that $\mX$ is minimal. Let $n\in \mathbb{Z}^+$ and for a given $a \in \mathcal{A}_n,$ let $I$ be an open interval in $X_n\setminus \{x_n\}$, the circle of  corresponding to $a.$ Let $x$ be any point of $\mX$ with $p_n(x)\in I$. As $\mX$ is minimal, the set of  return times of $x$ to $ p_n^{-1}(I)$ is relatively dense in $\mathbb{R}$, say every interval in $\mathbb{R}$ of length $L$ contains at least one return time to $ p_n^{-1}(I)$. As $\mathcal{S}$ is proper, the periods of the circles forming the $X_i$ grow without bound. Thus, there is an $m>n$ for which all the circles used to form $X_m$ have period greater than $L$. Now let $b$ be any symbol in  $\mathcal{A}_m$, with corresponding circle $X^m_j$. Now let $J=X^m_j\setminus \{x_m\}\subset X_m.$ As $p^{-1}_m(J)$ is open in $\mX$ and the orbit of $x$ is dense in $\mX,$ some point in the orbit of $x$ is contained in  $p^{-1}_m(J)$, and so the $p_m$ image of the orbit of $x$ covers $X^m_j$. As the period of the flow on $X^m_j$ is greater than $L$, there is a point $z\in p^{-1}_m(J)\cap p^{-1}_n(I)$ in the orbit of $x.$ Hence, $f_{n,m}(p_m(z))\in I$ and so $s_n \circ \cdots \circ s_{m-1}(b)$ contains the symbol $a.$ Thus, $\mathcal{S}$ is primitive as required.

\end{proof}

\subsection{Examples of expansions}\label{expansionexamp}

When a flow space is the tiling space given by the suspension of a minimal, aperiodic subshift, we can construct proper expansions very directly.  Let  $\mX$ then be the suspension of $(S,s)$, a minimal, aperiodic subshift of $(\cA_1^\mZ,s)$ for some finite alphabet $\cA_1$. We identify the base of the suspension $S\times \{0\}\subset S\times [0,1]/\!\!\sim$ with $S$ for simplicity of notation. Consider then any $a_1$ and $b_1$ in $\cA_1$ which occur as adjacent symbols in some element of $S$: for some $(x_i)_{i\in \mZ}\in S$ and for some $n\in \mZ$, $x_n=a_1$ and $x_{n+1}=b_1$. By the minimality of $S,$ the word $a_1b_1$ occurs infinitely often and with bounded gaps along any orbit of $S.$ Thus, there are only a finite number of possible words that lie between adjacent occurrences of $a_1b_1$. Let $\cA_2$ be the finite set of the possible such words, which we shall call `2-tiles', where we consider each such word as beginning with a $b_1$ and ending with an $a_1$. We can assume $\cA_2$ is of cardinality at least two by the aperiodicity.  Suppose $a_2b_2$ is a junction of 2-tiles that occurs in some element of $S.$ Again by minimality, such junctions occur  infinitely often, with bounded gaps, and so with a finite number of possible words that lie between adjacent occurrences to yield a finite alphabet $\cA_3$ of 3-tiles.  Repeating this recursively gives a sequence of finite alphabets $\cA_n$ encoding an infinite hierarchy of $n$-tiles. Let  $X_n$ be a wedge of circles in one to one correspondence with the $n$-tiles forming $\cA_n$. The bonding maps $X_n\to X_{n-1}$ are given by the $(n-1)$-tile `letters' making up each $n$-tile with intervals mapping linearly. The bonding maps are proper by construction. The projections $p_n: \mX \to X_n$ are defined naturally by the position of the point in the corresponding $n$-tile. Properness and the ever-increasing length of the words forming the $n$-tiles ensures that this does indeed form an expansion of $\mX$.

Note that the base point for the resulting expansion can be determined by the construction: it will be the point $x_\bullet=(x_i)_{i\in \mZ}$ in the base $S$ with initial  entries of $(x_i)_{i\geq 0}$ given by $b_n$ and with the terminal entries of  $(x_i)_{i< 0}$ given by $a_n$ for each $n$. As the lengths of the words forming $a_n$ and $b_n$ increase without bound, this completely determines $x_\bullet$. Turning this around,  we can arrange an expansion that is about any particular desired base point $x_\bullet=(x_i)_{i\in \mZ}\in S$ as follows.  Let $l_n$ be the word $x_{-n}x_{-n+1} \cdots x_{-1}$  and let $r_n$ be the word $x_0\cdots x_n$. We now perform the construction above declaring the $n$-tiles to be the finite words in $S$ that lie between occurrences of $l_nr_n$. 

When there is a clear way to form the hierarchy of $n$-tiles, the construction can be made more explicit, as we now see in the case of substitutions.

\begin{defn}\label{substsubshiftdefn}
Given a primitive substitution $\sigma \colon \cA \to \cA^*$ of a finite alphabet $\cA,$ the \emph{substitution subshift of $\sigma$}, denoted $S(\sigma)$, is the subshift of $(\cA^\mZ,s)$ consisting of those $x=(x_n)_{n\in \mZ}$ which satisfy the condition that any finite word $x_i\cdots x_{i+k}$ occurring in $x$ is also a subword of $\sigma^n(a)$ for some $a\in \cA$ and some $n\in \mZ^+$.
\end{defn}
It is well known that $(S(\sigma),s)$ is minimal when $\sigma$ is a primitive substitution, see, e.g., \cite{BG}.
Suppose then $\sigma$ is a primitive substitution on the finite alphabet $\cA_1$. Then after sufficient iteration (which we suppress for simplicity of notation), we can assume there are  $a,b \in \cA_1$ such that $\sigma(a)$ ends with an $a$, and $\sigma(b)$ starts with a $b$ and $ab$ occurs as a word in $S(\sigma)$; see, e.g., \cite[Lemma 4.3]{BG}. Now repeat the above general construction by declaring  the 2-tiles to be the words between occurrences of $ab$, and so in particular, any 2-tile will start with a $b$ and end with an $a$. In general, the $n$-tiles are the words in the $(n-1)$-tiles between occurrences of $\sigma^{n-1}(a)\sigma^{n-1}(b)$. The base point of the expansion is then $\sigma^\infty(a) \,\cdot \, \sigma^\infty(b)\in S(\sigma)$, where $\cdot$ delimits the terms $x_i$ with negative index from those with non-negative indices and $\sigma^\infty$ applied to an element of $\cA_1$ is determined by iteratively applying $\sigma.$  We now illustrate the technique with a particular example; it should be compared with the rewriting techniques of \cite{BD} Section 3.

\begin{ex}
The Thue-Morse substitution on the alphabet $\{a,b\}$ is given by 
$$\sigma\colon a\mapsto abba,\quad b\mapsto baab\,.$$
\end{ex}
In our construction $X_1$ will be copy of $W_2$. As $\sigma(a)$ begins and ends with an $a$, and  $\sigma(b)$ begins and ends with an $b$, we have four possible initial junctions we could choose:  $ab,\ aa,\ ba$ or $bb$.  Let us choose $aa$.

Considering a long word occurring in $S(\sigma)$, we have:
$$\cdots\ .abbaba.abba.ababbaba.ababba.abbaba.abba.ababba.abbaba.ababba.\cdots\, ,$$
where the dots mark the $aa$ junctions. We see we have four `words' giving us our 2-tiles, and so $X_2$ will be a copy of $W_4$. 
We specify the 2-tiles and calculate the substitution map on them as follows.
$$\begin{array}{lll}
A=abbaba&\buildrel\sigma\over\mapsto\quad abbabaabbaababbabaababba&=ABCD\\
B=abba&\buildrel\sigma\over\mapsto\quad abbabaabbaababba&=ABD\\
C=ababbaba&\buildrel\sigma\over\mapsto\quad abbabaababbabaabbaababbabaababba&=ACBCD\\
D=ababba&\buildrel\sigma\over\mapsto\quad abbabaababbabaabbaababba&=ACBD
\end{array}$$

The 3-tiles will be the words in the 2-tiles that lie between occurrences of $\sigma^2(a).\sigma^2(a)$'s. But of course these are $\sigma$ applied to the words that lie between occurrences of $\sigma(a).\sigma(a)$, and so are $\sigma(A)$,  $\sigma(B)$,  $\sigma(C)$  and  $\sigma(D)$. The expansion thus becomes constant, with $X_n=W_4$ for $n\geqslant 2$, and the bonding maps for $n\geqslant 2$ have symbolic representations as immediately above.

Thus, even though the initial substitution $\sigma$ is not proper, we arrive at a proper substitution. In general, if a primitive substitution on an alphabet with $n$ letters is additionally proper, then an expansion of the associated tiling space  can be constructing using a constant sequence for which each $X_i$ is $W_n$ and the bonding maps are represented symbolically by the substitution, see \cite{BD}. 


\section{Tiling Space Factors}\label{TilingSection}

In the study of the dynamics of homeomorphisms of the Cantor set, factors of the homeomorphism which can be represented as a subshift on a finite alphabet are often used. In this section we show how analogously we can represent any flow admitting a flow expansion as the inverse limit of a sequence of tiling space flows. 

In the case that the flow is minimal and equicontinuous, the tiling space flows in the limiting sequence are periodic. Any aperiodic tiling space flow is expansive, just as any aperiodic subshift is expansive. However, unlike the class of equicontinuous flows which are closed under inverse limits, the inverse limit of expansive flows need not be expansive, and indeed in many cases flows admitting a flow expansion are neither expansive nor equicontinuous. 

\begin{theorem}\label{tilingfactor}
Given a flow expansion $(\mX,x_\bullet)\approx_f\eX,$ the flow on $\mX$ is conjugate to the inverse limit of flows on tiling spaces obtained from the factor maps $p_n:\mX \to X_n$. Any flow space admitting an expansion is flow equivalent to the inverse limit of tiling spaces with factor maps as bonding maps.
\end{theorem}
\begin{proof}
Consider for a given $n\in \mZ^+$ the corresponding factor map $p_n:\mX \to X_n$ and the resulting factor map $\widehat{p}_n: \mX \to \cT(p_n)$ as given in Definition \ref{tilingfactordef}. It is clear that the bonding maps $f_n \colon X_{n+1}\to X_n$ induce factor maps  $\widehat{f}_n \colon \cT(p_{n+1})\to \cT(p_n)$ so that for $n,k \in \mZ^+$ we have  $\widehat{p}_n = \widehat{f}_{n,n+k+} \circ \widehat{p}_{n+k}$ due to the corresponding equality of the original mappings. Then the map $\widehat{h}: \mX \to \underleftarrow{\lim}\{\cT(p_1) \xleftarrow{\widehat{f}_1}\cT(p_2) \xleftarrow{\widehat{f}_2} \cT(p_3) \xleftarrow{\widehat{f}_3} \cdots\}$ given by $\widehat{h}(x)=(\widehat{p}_n(x))_{n\in \mZ^+}$ is a homeomorphism that conjugates the flow on $\mX$ with the inverse limit of the flows on $\cT(p_n)$.

The last statement of the theorem follows from Theorem \ref{gen exp to flow exp} as we can alter the flow on $\mX$ so that any expansion of $\mX$ becomes a flow expansion.
\end{proof}

Observe that the tiling spaces $\cT(p_n)$ can be periodic, and so it is possible to obtain aperiodic equicontinuous flows in the limit. In fact, if $\mX$ is minimal, each $\cT(p_n)$ will also be minimal, and the only equicontinuous minimal sets of a shift space are periodic. Thus, when $\mX$ is equicontinuous and admits an expansion, it is necessarily minimal and therefore each $\cT(p_n)$ is periodic. 

It has been proven in the context of Bratteli diagrams representing Cantor set homeomorphisms in \cite{DM}, with a subsequent alternate proof in \cite{H}, that any Bratteli diagram of finite rank represents an equicontinuous or expansive homeomorphism. However, there are expansive homeomorphisms of infinite rank. 

In our context, consider a flow space $\mX$ which has an  expansion $\mX \approx \eX$ which is bounded; i.e., there is a uniform bound $B$ on the number of circles used to form the $X_n$. Then the rational Cech cohomology $\check{H}^1(\mX;\mQ)=\displaystyle\lim_{n\to\infty}H^1(X_n;\mQ)$ has finite rank bounded above by $B$. An immediate consequence is that if we know a flow space $\mX$ has $\check{H}^1(\mX;\mQ)$ of infinite rank, then it cannot have a bounded expansion, for any choice of base point. Cases as such include the generic one dimensional canonical projection tilings with internal space of dimension at least 2. This follows from \cite{FHK} Chapter IV, Theorem 6.10. See Chapter IV Section 6 for a broader discussion of infinite generation in the cohomology of canonical projection tilings, yielding a wide class of tiling spaces that are not of bounded rank.

There is a distinction of the flows on flow spaces that has no direct analogue in Cantor set homeomorphisms.

\begin{defn}
Say the flow on a flow space admitting an expansion is \emph{tiling-like} if it is the limit of an inverse sequence of tiling space flows in which the bonding maps are factor maps.
\end{defn}

From the above Theorem and Theorem \ref{flow expansion}, if a flow space admits a factor map to some $W_n$ with a partial flow, then it is tiling-like. Conversely, if the flow is tiling-like, then it admits a factor map to a tiling space flow, and any tiling space flow admits a factor map to some $W\in \cW$. By Theorem \ref{flow expansion}, it therefore admits a flow expansion. It should be noted that the definition of tiling-like is not base point dependent.

\begin{question}
What is an intrinsic characterisation of the tiling-like flows?
\end{question}

It is known that tiling spaces arising from a non-Pisot substitution are generally weakly mixing \cite{S}, and so there are even tiling flows that are weakly mixing.  Thus, the characterisation of flows that are not tiling-like must involve a property stronger than (topologically) weakly mixing as any flow which is not  weakly mixing is a suspension flow.


\section{Rigidity of Flow Spaces}\label{SectRigidity}

With abelian topological groups, substitution tiling spaces and tiling spaces more generally, there are some established rigidity results that take various forms \cite{Sch},\cite{BS},\cite{K}. Here, we shall produce a result that is in the same spirit by showing that any flow equivalence of flow spaces is isotopic to a combinatorial isomorphism, which we now define.

\begin{defn}\label{gencombiso}
Given  flow spaces with expansions $(\mX,x_\bullet)\approx\eX$ and $(\mY,y_\bullet)\approx\eY$,  a \emph{combinatorial isomorphism} relative to these expansions is a flow equivalence $(\mX,x_\bullet)\to (\mY,y_\bullet)$  represented by a commutative diagram in which the maps $d_i$ and $u_i$ are positive maps:

\[
\begin{tikzcd}[font=\large] 
X_{m_1}  \arrow[swap]{d}{d_1} && X_{m_2}\arrow[swap]{d}{d_2}\arrow{ll}{f_{m_1,\,m_2}}&&X_{m_3}\arrow{ll}{f_{m_2,\,m_3}}\arrow[swap]{d}{d_3} &&  \cdots \arrow{ll}  \\ Y_{n_1} && Y_{n_2} \arrow{ll}{g_{n_1,\,n_2}}\arrow{llu}{u_1} &&Y_{n_3} \arrow{ll}\arrow{ll}{g_{n_2,\,n_3}}\arrow{llu}{u_2}  &&\cdots \arrow{ll}   
\end{tikzcd}
\]

\end{defn}

Observe that such a commutative diagram does determine continuous maps $d: (\mX,x_\bullet)\to (\mY,y_\bullet)$ and $u\colon (\mY, y_\bullet )\to (\mX, x_\bullet)$. In particular, for a given $\left( x_n\right)_{n\in \mZ^+}\in \underleftarrow{\lim}(X_i,f_i)$, its $d$ image is $\left( y_n\right)_{n\in \mZ^+} \in \underleftarrow{\lim}(Y_i,g_i)$, where for $n\in \mZ^+,$ $y_n = g_{n,\,m_k}\left(d_k(x_{n_k})\right)$, where $k$ satisfies $m_k>n$. Similarly for $u.$ By the commutativity of the diagram, the maps $d$ and $u$ are mutual inverses and thus homeomorphisms. As positive maps, $d_i$ and $u_i$ are represented up to homotopy by their symbolic representation. 
Observe that a combinatorial isomorphism yields an isomorphism of the inverse sequences of the expansions in $\overline{\cW}.$

If the flow spaces are equicontinuous and the approximating spaces $X_i$ and $Y_i$ in a flow expansion are circles, then the projection maps onto the $X_i$ and $Y_i$ are characters for a topological group structure, and the combinatorial isomorphism can be used to construct an isomorphism of the associated group structure. The following result is thus the analogue of a similar result that holds for certain classes of groups, including compact, connected, abelian groups such as solenoids \cite{Sch}. While our flow spaces generally do not have a true group structure, we will see in the following sections that there are associated algebraic structures with which the flow spaces can be identified up to flow equivalence. The following theorems are key to establishing the link between flow equivalences and isomorphisms of the corresponding algebraic structures.

\begin{theorem}\label{isotopic}
If $h:(\mX, x_\bullet) \to (\mY, y_\bullet )$ is a flow equivalence, then relative to \emph{any} given expansions of $(\mX, x_\bullet)$ and $(\mY, y_\bullet )$ there is an isotopic combinatorial isomorphism, yielding an isomorphism of the corresponding inverse sequences in $\overline{\cW.}$
\end{theorem}

\begin{proof}
Let $h:(\mX, x_\bullet) \to (\mY, y_\bullet )$ be a flow equivalence, and let $(\mX, x_\bullet)\approx\eX$ and $(\mY, y_\bullet )\approx\eY$ be any expansions.  By Theorem \ref{gen exp to flow exp}, we assume these are flow expansions without loss of generality.  We identify $\mX$ and $Y$ with the inverse limits of their expansions so that the diameters of the projections $p_i\, (q_i)$ to $X_i\,(Y_i)$ are bounded by $\frac{1}{2^i}$ as discussed in Section \ref{back}. We set $n_1=1$ and proceed to identify $m_1$ as in Definition \ref{gencombiso}. Consider a neighbourhood $U$ of $y_\bullet$ of the form $(-\epsilon,\epsilon)\,.\,q_1^{-1}(y_1)$. Each fibre $p_i^{-1}(x_i)$ contains $x_\bullet,$ and so by continuity $h\left( p_i^{-1}(x_i)  \right) \subset U$ for all $i$ sufficiently large.  Choose then $i$ and $\lambda>0$ so that $h\left( (-\lambda,\lambda)\,.\,p_i^{-1}(x_i)  \right) \subset U.$ In principle, it could happen that $h$ maps two points of $p_i^{-1}(x_i)$ to the same ``matchstick" $(-\epsilon,\epsilon)\,.\,z$ in $U.$ To avoid this, consider the pullback of the flow $\phi_{\mY}$ on $\mY$  to $\mX,$ $h^*\phi_{\mY}.$ 
Now let $\eta$ be the minimum return time to $p_i^{-1}(x_i)$ of the flow $h^*\phi_{\mY}.$ We now have the potentially smaller neighbourhood of $y$, $U'= (-\mu,\mu)\,.\,q_1^{-1}(y_1)$, where $\mu = \min \{\frac{\eta}{2}, \epsilon\}.$ As before, choose $j\geq i$ and $\tau>0$ so that $h\left( (-\tau,\tau)\,.\,p_{j}^{-1}(x_{j})  \right) \subset U'.$ As $p_{j}^{-1}(x_j)\subset p_i^{-1}(x_i),$ $h$ maps the points of $p_{j}^{-1}(x_j)$ injectively to the set of matchsticks in $U'$ and we can now isotope $h$ to a homeomorphism $h_1$ so that $h_1\left(  p_j^{-1}(x_j)   \right)  \subset q_1^{-1}(y_1).  $ (For the proof of Theorem \ref{germ}, observe that as $h_1$ maps segments of orbits to orbit segments, the return map to $p_j^{-1}(x_j)$ of the flow on $\mX$ is conjugate via the restriction of $h_1$ to the return map to $h_1\left(  p_j^{-1}(x_j)   \right) $ of the flow on $\mY.$)

The points of $h_1\left(  p_j^{-1}(x_j)   \right) $ can return via the flow $\phi_{\mY}$ to $ q_1^{-1}(y_1)$ multiple times before returning to $h_1\left(  p_j^{-1}(x_j)   \right) $. Recall that, as in the proof of Theorem \ref{flow expansion}, there is a partition of $q_1^{-1}(y_1)$ into clopen subsets $K(\mY)^1_1,\dots,K(\mY)_{\ell_1} ^1$ where $q_1$ maps the forward $\phi_{\mY}$ orbits of the points in the same clopen set $K(\mY)^1_k$ to the same circle in $Y_1$ until the return to $ q_1^{-1}(y_1)$. There is an analogous partition of $p_j^{-1}(x_j) $ into clopen sets $K(\mX)^j_1,\dots,K(\mX)_{k_j} ^j.$  Consider a partition of $h_1\left(  p_j^{-1}(x_j)   \right) $ into clopen sets $C_1,\dots,C_\ell$ that refines the partition $h_1\left(K(\mX)^j_1\right),\dots,h_1\left( K(\mX)^j_{k_j}\right)$ and where any two points in the same $C_i$ have the same itinerary with respect to the partition $K(\mY)^1_1,\dots,K(\mY)_{\ell_1} ^1$ of returns to $ q_1^{-1}(y_1)$ until the first return to $h_1\left(  p_j^{-1}(x_j)   \right) $.  Observe that $\mY$ is partitioned into the tubes $[0,r_i).C_i, $ where $r_i$ is the $\phi_{\mY}$ return time of points of $C_i$ to $h_1\left(  p_j^{-1}(x_j)   \right) $, and any given such tube is contained in a single tube $h_1 \left( [0,P(\mX)^{j}_{w_i}). K(\mX)^j_{w_i} \right)$,  where  $P(\mX)^{j}_{w_i}$ is the return time to the fibre $p_j^{-1}(x_j)$  as in Theorem \ref{flow expansion}. Now isotope $h_1$ to a homeomorphism $\tilde{h}$ that is the same as $h_1$ on $p_j^{-1}(x_j)$ and so that $\tilde{h}$ maps  $p_j$ fibres whose images lie within a single $C_i$ tube to $q_1$ fibres, which can be achieved $C_i$ tube by $C_i$ tube. As the $p_j$ fibres have images that can be split across multiple $C_i$ tubes, there is not necessarily an induced map $f: X_j \to Y_1$ satisfying $f\circ p_j = q_1 \circ \tilde{h}.$ Now, let $\delta>0$ be the minimum distance between any two $C_i$. Choose $\Delta>0$ so that points within $\Delta$ in $\mX$ are mapped by the (uniformly continuous) $\tilde{h}$ to points in $\mY$ of distance less than $\delta$ apart. Now choose $n_1\geq j$ so that $\frac{1}{2^{n_1}}<\Delta.$  By construction, $\tilde{h} \left(  p^{-1}_{n_1}(x_{n_1})\right)$ is contained in a single $C_i.$  As $C_i$ is contained in a single $\tilde{h} \left( K(\mX)^j_{w_i} \right)$ and $f_{j,n_1}$ is a factor map, for all $t\in [0,P(\mX)^{j}_{w_i})$ we have that $\tilde{h}\left(t.p^{-1}_{n_1}(x_{n_1})\right)$ is contained in the $\tilde{h}$ image of the same (portion of a) $p_j$ fibre and hence within a single $q_1$ fibre, leading to a well-defined continuous function $d_1: p_{n_1}\left(  [0,P(\mX)^{j}_{w_i})\, . \,p^{-1}_{n_1}(x_{n_1})\right) \to Y_1$ satisfying $d_1\circ p_j = q_1 \circ \tilde{h}$ for all points in $[0,P(\mX)^{j}_{w_i})\,. \,p^{-1}_{n_1}(x_{n_1})$   Now $P^1_{w_i}\,.\,p^{-1}_{n_1}(x_{n_1})$  is contained in $\tilde{h}\left( p_j^{-1}(x_j)\right)$ but not necessarily in a single $C_j$ as it is likely the image of several $p_{n_1}$ fibres. With $X^{n_1}_i$ a circle used to form $X_{n_1}$ as in Theorem \ref{flow expansion}, consider the point $P^1_{w_i}.x_{n_1}, $ where as before we regard $x_{n_1}$ as a point in $X^{n_1}_i$. The  $p_{n_1}$ fibre of this point will then have diameter less than $\Delta$ and so $\tilde{h}$ will map this fibre into a single $C_j$, and this allows us to extend  $d_1$ continuously to include the segment in $X^{n_1}_i$ between $P^1_{w_i}.x_{n_1}$ and the next (following the flow) point of $f^{-1}_{j,n_1}(x_j)$. We can then similarly continue to extend $d_1$ to include all of $X^{n_1}_i$. We then repeat the same process for each of the circles used to form $X_{n_1}$ to finally obtain a continuous map $d_1: X_{n_1} \to Y_{m_1}$ satisfying  $d_1\circ p_j = q_{m_1} \circ \tilde{h}$.  

We now proceed to identify $m_2$.  Consider the partition $\tilde{h}\left(K(\mX)^{n_1}_1\right),\dots,\tilde{h}\left(K(\mX)^{n_1}_{k_{n_1}}\right)$ of $ \tilde{h}\left( p^{-1}_{n_1}(x_{n_1})\right)$ and let $\delta>0$ be the minimum distance between elements of this partition. Choose $m_2>m_1$ so that $\frac{1}{2^{m_2}}<\delta.$ Then $q^{-1}_{m_2}(y_{m_2})$ is contained in a single $\tilde{h}\left(K(\mX)^{n_1}_i\right).$ As before,  this means there is a continuous map $u_1$ satisfying $u_1\circ q_{m_2} = p_{m_1}\circ \tilde{h}^{-1}$ defined initially on $t.q^{-1}_{m_2}(y_{m_2}) $ for non--negative $t$ less than the return time of $\tilde{h}\left(K(\mX)^{n_1}_i\right)$ to  $ \tilde{h}\left( p^{-1}_{n_1}(X_{n_1})\right).$ Then one can extend $u_1$ to all of $Y_{m_2}$ satisfying $u_1\circ q_{m_2} = p_{m_1}\circ \tilde{h}^{-1}$.

Using the same technique, one defines the $d_i$ and $u_i$ recursively for all $i$ so that we obtain a commutative diagram as in Definition \ref{gencombiso}. By construction, the resulting map $d: \mX \to \mY$ is identical to $\tilde{h}$ and $u$ is identical to  $\tilde{h}^{-1}.$
\end{proof}

Observe that while $\tilde{h}$ does respect the direction of the flow, it cannot be constructed to conjugate flows in general. As shown in \cite{CS}, there are changes to the lengths of tiles in non--Pisot substitution tiling spaces that change the conjugacy class of the flow, even allowing for a rescaling of the flow. The identification of the $X_{n_i}$ and $Y_{m_i}$ via the maps $d_i$ bears some similarity to changing tile lengths in the context investigated there.

It should be noted that there are a number of theorems of a similar spirit to Theorem \ref{isotopic}, e.g. \cite{M},\cite{R}, but we are able to obtain the stronger result of the isotopic homeomorphism due to the special nature of the spaces and expansions that we consider. There is a similar theorem for flow equivalences that do not preserve base points, where we can no longer require that the maps $d_i$ and $u_i$ preserve base points.

Maps that do not preserve the points about which one has taken an expansion, in general will not be even homotopic to a combinatorial isomorphism. This is true in a strong sense: the path components of flow spaces correspond with orbits, and asymptotic orbits must be mapped to asymptotic orbits. Thus, if the base point $x_\bullet$ in $\mX$ is along an orbit that is asymptotic to another orbit in $\mX$ but the base point $y_\bullet$ in $\mY$  is in an orbit that is not asymptotic to any other orbit, then there is no way to homotope a homeomorphism $h\colon \mX \to \mY$ to a combinatorial isomorphism as any such isomorphism preserves the base point. Thus, the following result is the most to be expected in general.

\begin{theorem}\label{unpointed}
If $h:\mX \to \mY$ is a flow equivalence, then relative to \emph{any} expansions $(\mX,x_\bullet)\approx\eX$ and $(\mY, y_\bullet )\approx\eY$  there is an isotopic flow equivalence that is represented by a diagram as follows in which the maps $d_i$ and $u_i$ preserve the direction of the flow, but in general may not take wedge point to wedge point.
\[
\begin{tikzcd}[font=\large] 
X_{m_1}  \arrow[swap]{d}{d_1} && X_{m_2}\arrow[swap]{d}{d_2}\arrow{ll}{f_{m_1,\,m_2}}&&X_{m_3}\arrow{ll}{f_{m_2,\,m_3}}\arrow[swap]{d}{d_3} &&  \cdots \arrow{ll}  \\ Y_{n_1} && Y_{n_2} \arrow{ll}{g_{n_1,\,n_2}}\arrow{llu}{u_1} &&Y_{n_3} \arrow{ll}\arrow{ll}{g_{n_2,\,n_3}}\arrow{llu}{u_2}  &&\cdots \arrow{ll}   
\end{tikzcd}
\]
\end{theorem}
\begin{proof}
 Given a flow equivalence $h:\mX \to \mY$  and expansions $(\mX,x_\bullet)\approx\eX$ and $(\mY, y_\bullet )\approx\eY$, we set $n_1=1$ and proceed to identify $m_1.$  The proof is very similar to that of Theorem \ref{isotopic}, and so we only highlight the portion of the proof where there is a significant difference. The role played by $q_1^{-1}(y_1)$ in the earlier proof is now played by $F=q_1^{-1}(q_1(h(x_\bullet))$. If $t$ is the minimal non--negative $t$ satisfying $t.\,q_1(h(x_\bullet))= y_1$, the sets $-t.\,K(\mY)^1_1 \cap F,\dots,-t.\,K(\mY)_{\ell_1} ^1 \cap F$ play the role previously played by the sets $K(\mY)^1_1,\dots,K(\mY)_{\ell_1} ^1.$ Similarly adjustments need to be made recursively; otherwise, the proof is essentially the same.

\end{proof}
Observe that by truncation, telescoping and re--indexing, we can obtain a commutative diagram of the following form in either the pointed or unpointed case:

\[
\begin{tikzcd}[font=\large] 
X_{1}  \arrow[swap]{d}{d_1} && X_{2}\arrow[swap]{d}{d_2}\arrow{ll}{f_{1}}&&X_{3}\arrow{ll}{f_{2}}\arrow[swap]{d}{d_3} &&  \cdots \arrow{ll}  \\ Y_{1} && Y_{2} \arrow{ll}{g_{1}}\arrow{llu}{u_1} &&Y_{3} \arrow{ll}\arrow{ll}{g_{2}}\arrow{llu}{u_2}  &&\cdots \arrow{ll}   
\end{tikzcd}
\]

Close examination of the proof of Theorem \ref{isotopic} reveals a close connection with the notion of germinal equivalence, which we now define.
\begin{defn}\label{germinal}
If $h_i:C_i \to C_i$ are minimal homeomorphisms of Cantor sets for $i=1,2$, we say that $h_1$ and $h_2$ are \emph{germinally equivalent} if for a given $i\in\{1,2\}$ and  clopen subset $K \subset C _i$ there is a corresponding $\delta_K>0$ 
so that for \emph{any} neighbourhood $N_{i+1} \subset C_{i+1}$ with diameter less than $\delta_K$,  there is a neighbourhood $N_i\subset K$ so that the return homeomorphism of $h_i$ to $N_i$ is conjugate to the return homeomorphism of $h_{i+1}$ to $N_{i+1}$, where $i+1$ is interpreted $\!\!\!\!\mod 2$. 
\end{defn}

\begin{theorem}\label{germ}
Minimal homeomorphisms $h_i:C_i \to C_i$ of Cantor sets for $i=1,2$ are germinally equivalent if and only if the suspensions of $h_1$ and $h_2$ are flow equivalent.  
\end{theorem}

\begin{proof}
Suppose that the homeomorphisms $h_i:C_i \to C_i$ of Cantor sets for $i=1,2$ are germinally equivalent. Observe that if $A\subset C_i$ is a clopen subset, then the suspension of the return homeomorphism of $h_i$ to $A$ is flow equivalent to the suspension of $h_i$, and so the conclusion follows directly. (Observe that for this direction we do not require minimality.)

Conversely, suppose that the suspension $\mX_1$ of $h_1$ is flow equivalent to the suspension $\mX_2$ of $h_2$. Given $i\in\{1,2\}$ and $K \subset C_i$ we then have that the suspension $\mX_{i+1}$ of $h_{i+1}$ is flow equivalent to the suspension $S_K$ of the return homeomorphism $\cR_K$ of $h_i$ to $K.$  $\mX_{i+1}$ is a quotient of $C_{i+1} \times [0,1]$ and $C_{i+1}$ can be identified with $C_{i+1} \times \{t\}$ for any $t\in [0,1)$ and the return map of the suspension flow to any such $C_{i+1} \times \{t\}$ is conjugate to $h_{i+1},$ and similarly for $S_K$ and $K \times [0,1]$.  Let $x\in C_{i+1}$ be any point, which we then regard as a point $x_\bullet$ in $C_{i+1} \times \{0\}\subset \mX_{i+1}.$  The map $p^{i+1}_1: \mX_{i+1} \to S^1= X^{i+1}_1$ given by $p^{i+1}_1\left( C_{i+1} \times \{t\} \right) = t \, \mod \, \mathbb{Z}$ (regarding $S^1$ as $\mathbb{R}/\mathbb{Z}$) is a factor map sending $x_\bullet $ to $\mathbf{0}$, which we regard as the base point in $X^{i+1}_1$. As shown in Theorem \ref{flow expansion}, this can be completed to a flow expansion about $x_\bullet,$ $(\mX_{i+1},x_\bullet)\approx\underleftarrow{\lim}\{X_1^{{i+1}} \xleftarrow{f_1} X_2^{i+1} \xleftarrow{f_2} X_3^{i+1} \xleftarrow{f_3} \cdots \cdots\}$. By hypothesis, there is a flow equivalence $h: \mX_{i+1} \to S_K.$ Let $\tau\in [0,1)$ be the uniquely determined element of $[0,1)$ such that $h(x_\bullet)\in K \times \{\tau\}.$ We then have the factor map $p^K_1: S_K \to S^1= X^{K}_1$ given by $p^K_1\left( C_K \times \{t\} \right) = t -\tau \, \mod \, \mathbb{Z}$, which sends $h(x_\bullet)$ to $\mathbf{0}$, which we regard as the base point in $X^{K}_1.$ Again, we can complete this to a flow expansion $(\mX_2,h(x_\bullet))\approx\underleftarrow{\lim}\{X_1^{(2)} \xleftarrow{g_1} X_2^{(2)}\xleftarrow{g_2} X_3^{(2)} \xleftarrow{g_3} \cdots \cdots\}$. By Theorem \ref{isotopic}  and the parenthetical remark in its proof, $h$ is isotopic to a combinatorial isomorphism, whose restriction provides a conjugacy of the return map of the flow on $\mX_{i+1}$ to a clopen subset $C$ of the $p^{i+1}_1$ fibre of $\mathbf{0}$  to the return map of the flow on $S_K$ to a clopen subset of the $p^{K}_1$ fiber of $\mathbf{0}$. As the return maps of the respective flows to these clopen subsets are conjugate to the return maps of the respective homeomorphisms of which they are suspesnsions to these clopen subsets, we have that $x$ is contained in a clopen set $A_x \subset C_{i+1}$ which has a corresponding clopen $B_x \subset K$ so that the return map of $h_{i+1}$ to $A_x$ is conjugate to return map of $\cR_K$ and hence $h_i$ to $B_x.$ Notice that by restriction of the conjugating homeomorphism,  for any (not necessarily clopen) neighborhood $N\subset A_x$  the return map to $N$ will be conjugate to the return map to a neighbourhood of $B_x.$ A Lebesgue number $\delta_K>0$ for the open cover of $C_{i+1}$ given by $\left\{\, A_x\,|\, x\in C_{i+1}\,\right\}$ then provides us with the required number as in Definition \ref{germinal} as any neighbourhood in $C_{i+1}$ of diameter less than $\delta_K$ will be contained in some $A_x.$

\end{proof}


\section{Complete invariant for pointed flow equivalence}\label{pointedsection}

In this section we introduce the complete invariant for pointed flow equivalence. The invariant will consist of equivalence classes of inverse sequences from the following category.
\begin{defn}\label{freep}
$\cP_F$ is the category with
\begin{itemize} 
\item Objects:  ordered pairs $(F_n, \{g_1,\dots,g_n\})$ consisting of a finitely generated free group $F_n$ of rank $n$, together with a chosen set of generators $\{g_1,\dots,g_n\}$ \emph{(positive generators)}, and 
\item Morphisms: $(F_m, \{h_1,\dots,h_m\})\to (F_n, \{g_1,\dots,g_n\})$ are  \emph{positive homomorphisms} consisting of homomorphisms  $F_m \to F_n$ which map the semigroup generated by $\{h_1,\dots,h_m\}$ into the semigroup generated by $\{g_1,\dots,g_n\}.$  
\end{itemize}
A positive homomorphism is \emph{proper} or \emph{primitive} if the corresponding symbolic map on the alphabets $\{h_1,\dots,h_m\}$ and $\{g_1,\dots,g_n\}$ is proper or primitive respectively. 

An element in $(F_n, \{g_1,\dots,g_n\})$ is \emph{positive} if it is in the semigroup generated by $\{g_1,\dots,g_n\}$ and it is \emph{negative} if it is in the semigroup generated by $\{g^{-1}_1,\dots,g^{-1}_n\}.$ The identity element is both positive and negative. 
\end{defn}

While we are referring to elements as positive and negative, we are not claiming this to be part of an ordering of the group in the formal sense.
For brevity, we will also denote an element of $\cP_F$ as $(F_n, \cG)$, where $\cG$ denotes the chosen set of generators.

Now we shall define a functor $\Pi$ from  $\overline{\cW}$ to  $\overline{\cP_F}$ that will ultimately give us the pointed invariant. 

\begin{defn}
Given $(X_i,f_i)$ in $\overline{\cW},$ we denote by $\Pi \left( (X_i,f_i) \right)$ the inverse sequence
$\left( (G_i,\{g(i)_1,\dots,g(i)_{k_i}\}), \pi_1(f_i) \right)$ in $\overline{\cP_F}$, where $G_i = \pi_1(X_i,x_i)$ and $g(i)_j$ is the homotopy class of a loop in $X_i$ based at $x_i$ that goes around the circle $X^i_j$ one time in the direction of the flow. 

For a morphism $\mu$ from $\left( X_i,f_i \right)\to \left( Y_i,g_i \right)$ in $\overline{\cW}$, we write $\Pi(\mu)$ for the morphism in $\overline{\cP_F}$  obtained by applying $\pi_1$ to all of the positive maps in $\cW$ in the ladder of the diagram representing $\mu.$
\end{defn}

Note that  the map $\pi_1(f_i)$ acts on the chosen sets of generators just as the symbolic representation of $f_i$ does on the alphabets given by the chosen generators.

\begin{lemma}\label{Pi}
If there is a flow equivalence $h\colon (\mX,x_\bullet) \to (\mY,y_\bullet) $, then for any expansions $(\mX,x_\bullet)\approx\eX$ and $(\mY, y_\bullet )\approx\eY$ we have that  $\Pi \left( (X_i,f_i) \right)$  is isomorphic to  $\Pi \left( (Y_i,g_i) \right)$ in $\overline{\cP_F}$.
\end{lemma}

\begin{proof}
Suppose $h\colon (\mX,x_\bullet) \to (\mY,y_\bullet) $ is a flow equivalence and we are given corresponding expansions. Then by Theorem \ref{isotopic}  the inverse sequences $ (X_i,f_i) $  and $(Y_i,g_i) $ are isomorphic in $\overline{\cW}$. It then follows that by applying $\Pi$ to the isomorphism of  $(X_i,f_i) $  and $(Y_i,g_i) $ represented by the zigzag diagram directly yields an isomorphism between  $\Pi \left( (X_i,f_i) \right)$ and  $\Pi \left( (Y_i,g_i) \right)$ in $\overline{\cP_F}$.
\end{proof}
Observe that while we are not requiring minimality, the lemma is vacuous if either $\mX$ does not admit an expansion about $x_\bullet$ or $\mY$ does not admit an expansion about $y_\bullet$.

As a result of this lemma, the following notion is well-defined.

\begin{defn}
If the flow space $\mX$ admits an expansion about $x_\bullet,$ then $\p\left(\mX,x_\bullet\right)$ is the isomorphism class in $\overline{\cP_F}$ of $\Pi \left( (X_i,f_i) \right)$ for any sequence $ (X_i,f_i)$ from an expansion about $x_\bullet.$
\end{defn}

\begin{theorem}\label{mainpointedtheorem}
Suppose the flow spaces $\mX$ and $\mY$ admit expansions about the points $x_\bullet$ and $y_\bullet$ respectively. Then there is a flow equivalence $(\mX,x_\bullet) \to (\mY,y_\bullet)$ if and only if $\p(\mX,x_\bullet) = \p(\mY,y_\bullet).$
\end{theorem}
\begin{proof}
If there is a flow equivalence $(\mX,x_\bullet) \to (\mY,y_\bullet)$ then it follows from Lemma \ref{Pi} that $\p(\mX,x_\bullet) = \p(\mY,y_\bullet).$ 

Conversely, suppose  $\p(\mX,x_\bullet) = \p(\mY,y_\bullet)$. Let
$(\mX,x_\bullet)\approx\eX$ and $(\mY, y_\bullet )\approx\eY$ be expansions.  By our hypothesis there is an isomorphism between $\Pi\left( (X_i,f_i)\right)$ and $\Pi\left( (Y_i,g_i)\right)$. The zigzag diagram representing this isomorphism translates directly to a commutative zigzag diagram in $\overline{\cW}$ between inverse sequences $(X'_i,f'_i)$  and $(Y'_i,g'_i)$ in $\overline{\cW}$ which have the same symbolic representations as $(X_i,f_i)$ and $(Y_i,g_i)$ respectively. By Lemma  \ref{seq rep equiv},  there is a base point preserving flow equivalence from the limits of $(X'_i,f'_i)$  and $(Y'_i,g'_i)$ to $(\mX,x_\bullet)$ and $(\mY,y_\bullet)$ respectively, and the result follows. 

\end{proof}

It is important to realise that if $x_\bullet,x'_\bullet \in \mX$ are not in the same orbit under the action of the homeomorphism group of $\mX,$ then Theorem \ref{mainpointedtheorem} tells us that we must have  $\p(\mX,x_\bullet) \neq \p(\mX,x'_\bullet)$, showing these objects really do depend on the base points considered. The solenoids (suspensions of adding machines) and circles are the only connected flow spaces on which the homeomorphism group acts transitively, \cite{AHO}. For other spaces there can be uncountably many different values of $\p(\mX,x_\bullet)$ as $x_\bullet$ ranges over $\mX$.

\section{The positive trope: a complete invariant}\label{mainunpointedtheorem}

As just noted, like  the fundamental group itself, $\p$ is very much dependent on the choice of base point. As the issue is with base points varying across different path components,  we clearly cannot consider paths in $\mX$ between two points of interest. However, the approximating spaces $X_i$ in an expansion \emph{are} path connected, and we can consider the images of two base points in $\mX$ in any such expansion. While there will be paths between such images in $X_i$, it is unreasonable to suppose such paths can be chosen coherently, lifting to one another, unless the two base points in $\mX$ lie in the same path component (in which case it is easily shown that values of $\p$ for each base point will be canonically isomorphic). 

This line of thought leads to the notion of the \emph{trope} as pioneered by Fox \cite{F1}. We will see that instead of conjugating by a single group element to measure the effect of a change of base point, as is the case for $\pi_1(X,x)$ for a path connected space and two possible base points $x$, we will need to conjugate by a sequence of elements. We will adjust Fox's idea of fundamental trope to take into account the orientations induced by the flow, leading us to our notion of \emph{positive trope}. 

For this part we shall  only consider minimal flow spaces so as  to avoid needing to treat different points differently. Additionally, we shall only consider aperiodic flow spaces as any minimal flow space that is not aperiodic is simply a single periodic orbit.

\subsection{The Positive Trope Relation.}

If $f\colon G \to H$ is a group homomorphism and $a\in H,$ then $\cc_a h$ denotes the homomorphism acting $\cc_ah(x)=a^{-1}h(x)a$ for $x\in G$.

\begin{defn}
The positive homomorphisms $f,g\colon G\to H$ in $\cP_F$ are {\em positive trope related}, denoted $f\s g$,  if there is an element $a\in H$ such that $f=\cc_ag$. 
\end{defn}

The following shows that $\s$ is an equivalence relation that is preserved by composition. 

\begin{lemma}\label{transcompos}
The relation $\s$ on positive homomorphisms in $\cP_F$ is reflexive, symmetric, transitive and compositive.
\end{lemma}

\begin{proof}

By choosing $a=1$ we see $f\s f$. Likewise, if $f\s g$ by $f=\cc_ag$ then $g=\cc_{a^{-1}}f$ which proves symmetry. For transitivity, if $f=\cc_ag$ and $g=\cc_bh$ then
$$f(x)\ =\ a^{-1}g(x)a\ =\ a^{-1}b^{-1}h(x)ba\ =\ (ba)^{-1}h(x)(ba)\,.$$
For compositive, suppose $f=\cc_ag\colon G\to H$ and let $r\colon K\to G$ and $s\colon H\to L$ be homomorphisms. Then for $x\in G$ and $y\in K$,
$$fr(y)\ =\ a^{-1}(gr)(y) a\qquad\mbox{ and }\qquad sf(x)\ =\ s(a^{-1}g(x)a)\ =\ (s(a))^{-1}(sg)(x) (s(a))\,.$$
\end{proof}

While we cannot require the conjugating element $a$ to be positive by the symmetry condition, we do have the following.

\begin{lemma}\label{plusorminus}
Suppose $f,g\colon G\to H$ are positive homomorphisms in $\cP_F$ satisfying $f=\cc_a g$. Then $a\in H$ is either  positive or  negative.
\end{lemma}

\begin{proof}
Suppose $\{e_i\}=\cG$ are the chosen positive generators of $G$ and let $w_i=g(e_i)$. As $g$ is positive, these are positive elements of $H$. As $f$ is also positive, $f(e_i)=\cc_ag(e_i)=a^{-1}w_ia$ must be positive for all $i$. Suppose, after all possible cancelling,
$$
a=p_1n_1\cdots p_rn_r\qquad\mbox{ where one or both of $p_1$ and $n_r$ may be 1}
$$
and where the $p_j$ are positive, and the $n_j$ negative. Then we must have 
$$\cc_a(w_i)\ =\ n_r^{-1}p_r^{-1}\cdots n_1^{-1}p_1^{-1}w_ip_1n_1\cdots p_rn_r$$
a positive word.

If $p_1\not=1$ then, looking at the last $2r-1$ letters,  $(n_1\cdots p_rn_r)=1$ since otherwise the element $\cc_a(w_i)$ could not be positive. As we were writing $a$ in its lowest terms, $r$ must therefore be 1 and $a=p_1$. 

On the other hand, for similar reasons, if $n_r\not=1$ then $(p_1n_1\cdots p_r)=1$. So again $r=1$ and in this case $a=n_1$. 
\end{proof}

In the following, for positive words $x,y$ we say $x$ is a {\em prefix} of $y$ if $y=xt$ for some positive word $t$, and $x$ is a {\em suffix} of $y$ if $y=sx$ for some  positive word $s$. We include the possibility that $s$ or $t$ may be the trivial word. In the case that $t\not=1$, say $x$ is a {\em strict prefix\/} of $y$, and if $s\not=1$, a {\em strict suffix\/} of $y$.

The argument above shows that the only $\cc_a$ that can preserve positivity acts on each word $w_i$ by a cyclic permutation of its letters.  

Depending on the exact nature of the set of words $w_i$ (the image under $g$ of the positive generators of $G$), more can be said. For example

\begin{cor}\label{cor2dim}
Let $f,g\colon G\to H$ be  positive homomorphisms in $\cP_F$ related by $f=\cc_a g$. Denoting $w_i=g(e_i)$ as above, if $w_i\not=w_j$ for some $i,j$ where neither $w_i$ nor $w_j$ are suffixes or prefixes of the other, then 
\begin{itemize}
\item If $a$ is positive, then $a$ is a prefix of \emph{every} $w_i$. Then $\cc_a(w_i)$ is obtained from $w_i$ by removing the prefix $a$ and adding it as a suffix.
\item If $a$ is negative, then $a^{-1}$ is a suffix of \emph{every} $w_i$.Then $\cc_a(w_i)$ is obtained from $w_i$ by removing the suffix $a^{-1}$ and adding it as a prefix.
\hfill $\square$
\end{itemize}
\end{cor}

At the other extreme, when there is some common word $w\in H$ such that $f(e_i)=w$ for all $i$, there is the possibility that $a$ is a longer word than $w$, which we examine in the next result. When this does happen, we can still arrange the word $a$ to be a suffix or prefix of $w$.

\begin{lemma}\label{lem1dim}
Suppose $f,g\colon G\to H$ are positive homomorphisms in $\cP_F$ related by $f=\cc_a g$ for some positive word $a$. Suppose also that there is a word $w\in H$ such that  $g(e_i)=w$ for all $i$. Then $a=w^nv$ for some positive word $v$ which is a strict prefix of $w$ and $n$ is a non-negative integer. In this case we also have $f=\cc_v g$.

Similarly, if $a$ is negative, we must have $a=uw^{-n}$ for some negative word $u$ which is a strict suffix of $w$ and again $n$ is a non-negative integer. In this case  we also have $f=\cc_u g$.
\end{lemma}

\begin{proof}
We consider the case where $a$ is positive. The negative case is analogous. Suppose $r$ is the number of letters in the positive word $w$, and factor (by positive subwords, and so in lowest possible terms)
$$a\ =\ a_na_{n-1}\cdots a_1v$$
where each $a_i$ has exactly $r$ letters, and $v$ is the remainder. The integer $n$ is non-negative. Then for $\cc_a(w)$ to be positive, we require
$$\cc_a(w)\,=\ v^{-1}a_1^{-1}\cdots a_{n-1}^{-1}a_n^{-1}wa_na_{n-1}\cdots a_1v$$
to be positive. This can only happen if $a_n=w$. This product is now
$$v^{-1}a_1^{-1}\cdots a_{n-1}^{-1}a_na_{n-1}\cdots a_1v\,.$$
For this to be positive, we need $a_{n-1}=a_n$. Continuing, and  examining how the negative terms in the middle can cancel, we see we must have 
$$a_1=a_2=\cdots a_n=w$$
and so $a=w^nv$ as claimed. 

Finally, note that this computation also shows that $\cc_a(w)=\cc_v(w)$.
\end{proof}

Similar analyses can be carried out in more complex situations when various $w_i$'s are suffixes and prefixes of others.

To streamline our presentation, we introduce the following.

\begin{defn}
A \emph{primitive, proper sequence} in $\overline{\cP_F}$ is a sequence $\left( (G_i,\cG_i), f_i \right)$ for which the symbolic sequence induced by the action of the $f_i$ on the alphabets $\cG_i$ is both primitive and proper. For short, we refer to such a sequence as a \emph{\pps} and we will denote it as $G_\bullet$.
\end{defn}

By the results of Theorems \ref{symbolicexpansion} and \ref{minimal}, any expansion of an aperiodic, minimal flow space will have a corresponding inverse sequence $(X_i,f_i)$ for which $\Pi\left( (X_i,f_i) \right)$ is a \pps. Conversely, by the same results, any \pps\ will be the  $\Pi$ value of an inverse sequence derived from an expansion of an aperiodic, minimal flow space.

The positive trope relation $\s$ on homomorphisms induces an equivalence relation, also denoted $\s$, on the objects of $\cP_F$, by declaring $(G,\cG)\s(H,\cH)$ if and only if there are positive homomorphisms $f\colon G \to H$ and $g\colon H \to G$ satisfying $g\circ f \s \text{id}_G$ and $f\circ g \s \text{id}_H$. This extends to $\overline{\cP_F}$ as follows.

\begin{defn}
Two \pps's $G_\bullet$ and $H_\bullet$ are {\em positive trope equivalent}, denoted $G_\bullet\s H_\bullet$,  when there is a diagram of positive homomorphisms
\[
\begin{tikzcd}[font=\large] 
G_{m_1}  \arrow[swap]{d}{d_1} && G_{m_2}\arrow[swap]{d}{d_2}\arrow{ll}{\sigma_{m_1,\,m_2}}&&G_{m_3}\arrow{ll}{\sigma_{m_2,\,m_3}}\arrow[swap]{d}{d_3} &&  \cdots \arrow{ll}  \\ H_{n_1} && H_{n_2} \arrow{ll}{\tau_{n_1,\,n_2}}\arrow{llu}{u_1} &&H_{n_3} \arrow{ll}\arrow{ll}{\tau_{n_2,\,n_3}}\arrow{llu}{u_2}  &&\cdots \arrow{ll}   
\end{tikzcd}
\]
in which the triangles commute up to conjugation. That is, for each $m_i$ there is an element $g_i\in G_{m_i}$ such that $u_id_{i+1}=\cc_{g_i}\sigma_{m_i,m_{i+1}}$, and for each $n_j$ there is an element $h_j\in H_{n_j}$ such that $d_ju_j=\cc_{h_j}\tau_{n_j,n_{j+1}}$.  Denote by $PT(G_\bullet)$  the $\s$ equivalence class of $G_\bullet$.
\end{defn}
We will refer to such a diagram as a \emph{conjugate zigzag} diagram, or \emph{czz} for short. Observe that if $G_\bullet$ and $H_\bullet$ are isomorphic in $\overline{\cP_F}$, then $G_\bullet \s H_\bullet$: we simply choose  $g_i$ and $h_j$ to be the identity element for each $i$ and $j$. In particular, $G_\bullet$ is positive trope equivalent to any telescoped or truncated version of $G_\bullet.$ Thus, by telescoping and truncation, we can assume up to isomorphism in $\overline{\cP_F}$ that any czz takes the following simplified form:

\begin{equation}\label{redzz}
\begin{tikzcd}[font=\large] 
G_{1}  \arrow[swap]{d}{d_1} && G_{2}\arrow[swap]{d}{d_2}\arrow{ll}{\sigma_{1}}&&G_{3}\arrow{ll}{\sigma_{2,3}}\arrow[swap]{d}{d_3} &&  \cdots \arrow{ll}  \\ H_{1} && H_{2} \arrow{ll}{\tau_{1}}\arrow{llu}{u_1} &&H_{3} \arrow{ll}\arrow{ll}{\tau_{2,\,3}}\arrow{llu}{u_2}  &&\cdots \arrow{ll}   
\end{tikzcd}
\end{equation}
where $u_nd_{n+1}=\cc_{g_n}\sigma_{n,n+1}$ and $d_nu_n=\cc_{h_n}\tau_{n,n+1}$ for some elements $g_n\in G_n$ and $h_n\in H_n$. We will refer to such a diagram as a \emph{simplified czz}.

Recall that by Lemma \ref{plusorminus} the words $h_{n}$ and $k_{n}$ are either positive or negative.

This establishes that the following concept is well-defined and independent of the expansion used as two expansions about the same point lead to isomorphic sequences in $\overline{\cP_F}$.

\begin{defn}
For a minimal, aperiodic flow space $(\mX,x_\bullet)$, the {\em positive trope class}, is $PT(G_\bullet)$, where $G_\bullet$ is $\Pi\left((X_i,f_i)\right)$ for any expansion $(\mX,x_\bullet)\approx \eX$, and the positive trope class is denoted $PT(\mX,x_\bullet)$.
\end{defn}

By construction, $PT(\mX,x_\bullet)$ as a class contains $\p(\mX,x_\bullet)$, and so we see that $PT(\mX,x_\bullet)$ is an (incomplete) invariant of pointed flow equivalence. We now show that the class $PT(\mX,x_\bullet)$  is actually independent of the base point $x_\bullet$.

\begin{theorem}\label{unpointed invariance}
If $\mX$ and $\mY$ are flow equivalent, minimal, aperiodic flow spaces, then for any choices of base points $x_\bullet\in \mX$ and $y_\bullet\in \mY$, we have $PT(\mX,x_\bullet)=PT(\mY,y_\bullet)$.
\end{theorem}

\begin{proof}
Suppose there is a flow equivalence $\mX \to \mY$. By Theorem \ref{unpointed} there are expansions $(\mX,x_\bullet)\approx \eX$ and  $(\mY,y_\bullet)\approx \eY$ for which we have a commutative diagram of maps representing a flow equivalence
\[
\begin{tikzcd}[font=\large] 
X_{1}  \arrow[swap]{d}{d_1} && X_{2}\arrow[swap]{d}{d_2}\arrow{ll}{f_{1}}&&X_{3}\arrow{ll}{f_{2,3}}\arrow[swap]{d}{d_3} &&  \cdots \arrow{ll}  \\ Y_{1} && Y_{2} \arrow{ll}{g_{1}}\arrow{llu}{u_1} &&Y_{3} \arrow{ll}\arrow{ll}{g_{2,\,3}}\arrow{llu}{u_2}  &&\cdots \arrow{ll}   
\end{tikzcd}
\]
where the maps $d_n$ and $u_n$ do not necessarily preserve base points but do preserve the direction of the partial flows. To define the homomorphisms $(d_n)_*\colon\pi_1(X_n,x_n)\to \pi_1(Y_n,y_n)$ and $(u_n)_*\colon\pi_1(Y_{n+1},y_{n+1})\to \pi_1(X_n,x_n)$ we must choose paths $p_n$ and $q_n$ where $p_n$ is a path in $Y_n$ from $d_n(x_n)$ to $y_n$ and $q_n$ is a path in $X_n$ from $u_n(y_{n+1})$ to $x_n$. As is classically known, the homomorphism $(d_n)_*$ is then given on a homotopy class of loop $\lambda$ in $X_n$ based at $x_n$ by $(d_n)_*[\lambda]=[p_n^{-1}\cdot d_n(\lambda)\cdot p_n]$,\footnote{The notation here is meant to indicate that the image of the homotopy class of the loop $\lambda$ under the homomorphism $(d_n)_*$ is the homotopy class of the loop that starts at $y_n$, goes along the path $p_n$ backwards to $d_n(x_n)$, then describes the loop $d_n(\lambda)$ and finally returns to $y_n$ by $p_n$.} and likewise for $u_n$ using the path $q_n$. These homomorphisms depend on the homotopy classes of the paths $p_n$ and $q_n$, and choices of different paths will  give homomorphisms that differ by a conjugation. We also wish that the homomorphisms $(u_n)_*$ and $(d_n)_*$ are positive, so we make the choice of path $p_n$ that travels from $d_n(x_n)$ to $y_n$ in the negative\footnote{In fact we could equally have chosen the path $p_n$ to go between those points via the \emph{positive} direction, and likewise for $q_n$. For some specific constructions there can even  be merit in choosing say  the $p_n$'s to be positive and the $q_n$'s negative, or vice-versa.} direction around the circle on which $d_n(x_n)$ lies, and likewise for $q_n$. If $d_n(x_n)=y_n$ we  just choose $p_n$ as the trivial path, etc. This ensures that if $\lambda$ is a positive loop in $X_n$, then $(d_n)_*[\lambda]$ is also positive, and similarly for $(u_n)_*$.

Explicitly, for $[\lambda]\in\pi_1(Y_{n+1},y_{n+1})$
$$\begin{array}{rl}
(d_n)_*(u_n)_*[\lambda]\ =&[p_n^{-1}\cdot (d_n(q_n))^{-1}\cdot (d_n u_n(\lambda))\cdot (d_n(q_n))\cdot p_n]\\
=&[r_n^{-1}\cdot (g_{n,n+1}(\lambda))\cdot r_n]
\end{array}$$
where $r_n=(d_n(q_n))\cdot p_n$ originates at $d_n(u_n(y_{n+1}))=y_n$ and terminates where $p_n$ terminates, $y_n$. Thus, $r_n$ is a loop based at $y_n$ and 
$$(d_n)_*(u_n)_*\ =\ \cc_{[r_n]}(g_{n,n+1})_*$$
as required. For the upper triangles the calculation is similar. 
Thus the triangles commute up to conjugation and altogether form a czz diagram. Hence $PT(\mX,x_\bullet)=PT(\mY,y_\bullet)$.

\end{proof}

If we take a single flow space $\mX$, but with two choices of base point, $x_\bullet$ and $x_\bullet'$ say, then the identity map on $\mX$ is a flow equivalence, and so $PT(\mX,x_\bullet)=PT(\mX,x_\bullet')$. 

\begin{cor}
For a minimal, aperiodic flow space $\mX$, $PT(\mX,x_\bullet)$ is independent of the choice of base point $x_\bullet$. 
\end{cor}

\noindent Accordingly, we denote the class by $PT(\mX)$, a flow equivalence invariant of $\mX$.

\medskip
To complete the proof that $PT$ is a complete invariant of flow equivalence, we require the following lemma.

\begin{lemma}\label{decompose}
Given a simplified czz diagram as in Diagram \ref{redzz}, the subdiagrams
$$
\begin{array}{rcl}
G_1&\buildrel \sigma_{1,n}\over\longleftarrow &G_n\\
\downarrow_{d_1}&&\downarrow_{d_n}\\
H_1&\buildrel \tau_{1,n}\over\longleftarrow &H_n
\end{array}
\qquad\mbox{and}\qquad
\begin{array}{rcl}
G_1&\buildrel \sigma_{1,n}\over\longleftarrow &G_n\\
\uparrow_{u_1}&&\uparrow_{u_n}\\
H_2&\buildrel \tau_{2,n+1}\over\longleftarrow &H_{n+1}
\end{array}
$$
commute up to conjugation; that is, there are elements $\rho_n\in H_1$ and $\zeta_{n}\in G_1$ such that
$$\tau_{1,n}d_n\ =\ \cc_{\rho_n}d_1\sigma_{1,n}\qquad\mbox{and}\qquad
\sigma_{1,n}u_n\ =\ \cc_{\zeta_n}u_1\tau_{2,n+1}\,.$$
\end{lemma}

\begin{proof}
The proof makes multiple use of the compositive nature of $\s$. For any given $n$ we have $u_{n}d_{n+1}\s \sigma_n$, and so applying $d_n$ we obtain $d_nu_{n}d_{n+1}\s d_n\sigma_n$, whence $\tau_nd_{n+1}\s d_n\sigma_n$. We now establish the diagram on the left for $n=3$ as we have already established it for $n=2.$
As $d_2\sigma_2 \s \tau_2d_3$, applying $u_1$ we have $u_1d_2\sigma_2 \s u_1\tau_2d_3$, whence $\sigma_1\sigma_2 \s u_1\tau_2d_3.$ Now applying $d_1$ we obtain 
\[
d_1\sigma_{1,3}\, \s d_1u_1\tau_2d_3 \s \tau_{1,3}\,d_3
\]
as required. Applying this argument recursively, we obtain the desired relations for the left diagram, and analogous arguments apply to the right hand diagram.
\end{proof}
Similar techniques can be applied to show that any two paths through the diagram originating and terminating at the same places will yield positive trope equivalent homomorphisms.

\begin{theorem}\label{ThmA}
Suppose $\mX$ and $\mY$ are minimal, aperiodic flow spaces satisfying $PT(\mX)=PT(\mY)$. Then $\mX$ and $\mY$  are flow equivalent.
\end{theorem}

\begin{proof}
Suppose $\mX$ and $\mY$ are minimal, aperiodic flow spaces satisfying $PT(\mX)=PT(\mY)$ with corresponding expansions $\mX \approx \eX$ and $\mY \approx \eY$, which we assume to be flow expansions, with $k_i$ and $\ell_i$ denoting the numbers of circles in $X_i$ and $Y_i$ respectively. By truncating and telescoping as necessary, we assume without loss of generality that $G_\bullet= \Pi ((X_i,f_i))$ and $H_\bullet= \Pi ((Y_i,g_i))$ satisfy a simplified czz as in diagram (\ref{redzz}). 

Consider the tiling space factors as in Theorem \ref{tilingfactor}, which will all be minimal by the minimality of $\mX$ and $\mY$. From the positive homomorphisms $d_i$ in the czz yielding the positive trope equivalence of $G_\bullet$ and $H_\bullet$, we now construct continuous mappings $\widetilde{d}_i$ from $\cT(p_i)$ to the global tiling space  $\cT(Y_i)$ as in Definition \ref{global tiling space}. We identify a point in a tiling space with its unique representative in the set
\[
\left\{\,(t,z) \colon 0\leq t < c(z)\,,\,z\in Z\right\},
\]
where $Z$ is the corresponding subshift. We first define  $\widetilde{u}_i$ on the base 
\[
\left( \{0\}\times\{a_1^i,\dots,a^i_{k_i}\}^\mZ\right) \cap \cT(p_i),
\] which will be mapped into the base $\{0\}\times\{b_1^i,\dots,b^i_{\ell_i}\}^\mZ \subset \cT(Y_i)$, and we identify each base with the corresponding subshift in the natural way.   Then, on the base, $\widetilde{d}_i$ is given by
\[
\dots x_{-2}x_{-1}\,\cdot  x_0x_1x_2\dots\, \mapsto \, \dots d_i(x_{-2})d_i(x_{-1})\,\cdot d_i( x_0) d_i(x_1)d_i(x_2)\dots \,,
\]
where $\cdot$ delimits the terms with negative indices from those with non-negative indices, and the terms in the image sequence are determined by concatenating the $d_i$ images. Then $\widetilde{d}_i$ is extended to the arc above a given sequence $(x_n)_{n\in\mZ}$ by mapping it linearly onto the arcs above $\dots d_i(x_{-2})d_i(x_{-1})\,\cdot d_i( x_0)d_i(x_1)d_i(x_2)\dots $  and its first $\ell$ shifts, where $\ell$ is the number of symbols in $d_i(x_0)$. There is no reason to expect  $\widetilde{d}_i$ to be a factor map. While the image of $\widetilde{d}_i$  is minimal due to its continuity and the minimality of $\cT(p_i)$, we do not know a priori that this image coincides with $\cT(q_i),$ but we now show they do coincide.

By Lemma \ref{decompose}, $ \pi_1(g_{1,n}) \circ d_n=\cc_{\rho_n} \circ d_1 \circ  \pi_1(f_{1,n})$ for some $\rho_n$. Comparing the induced map $\widetilde{\cc_{\rho_n} \circ d}_1$ on $\cT(p_1)$ with $\widetilde{u}_1$, for a point $(x_n)_{n\mZ}$ in the base, the sequence $\widetilde{\cc_{\rho_n} \circ d}_1((x_n)_{n\mZ})$ is simply a shift of $\widetilde{d}_1((x_n)_{n\mZ})$ as successive conjugations cancel but result in shifts, where the number of shifts is determined by the length of $\rho_n$ and the direction of the shift is determined by whether $\rho_n$ is positive or negative, see Lemma \ref{plusorminus}. Thus, any finite sequence of consecutive terms of symbols (or word) that occurs in   $\widetilde{u}_1((x_i)_{i\geq 0})$ is a word appearing in the image of $ \pi_1(g_{1,n})$ for sufficiently large $n$. But these are precisely the words appearing in the minimal subshift in the base of $\cT(q_1).$  Hence, these minimal subshifts coincide and the image of $\widetilde{u}_1$ is indeed $\cT(q_1).$  Similarly, for each $i$ we have that $\widetilde{d}_i(\cT(p_i))= \cT(q_i)$. Identifying $\mX$ and $\mY$ with the inverse limits of $(\cT(p_i),\widetilde{f}_i)$ and $(\cT(q_i),\widetilde{g}_i)$, we now have a sequence of maps $\widetilde{d}_i\colon \cT(p_i) \to \cT(q_i)$ which at this point do not necessarily form a commutative ladder to determine a map $\mX \to \mY$ due to the need to conjugate in the czz. However, we know that the effect of conjugation results in a shift, which can be compensated for recursively (bearing in mind that the $\widetilde{g}_i$ are factor maps) by a translation (time $t$ map of the flow) according to the length of the tiles associated to the conjugating element after the application of the  $\widetilde{d}_i$, with the translations chosen recursively so that we ultimately have a commutative ladder of maps determining a map $d\colon \mX \to \mY.$ An analogous construction yields a map $u\colon \mY \to \mX.$ The compositions of $u$ and $d$ yield a time $t$ map on the respective flow spaces, and so we can conclude that they are both flow equivalences. 

\end{proof}

\begin{cor}\label{PTmain}
Let $\mX$ and $\mY$ be minimal, aperiodic flow spaces. Then $\mX$ and $\mY$ are flow equivalent if and only if $PT(\mX) =PT(\mY).$
\end{cor}

As we now will show, given an expansion of a flow space $\mX$, the proof of Theorem \ref{ThmA} allows us to see how to identify expansions for $\mX$ about alternative base points. 

\begin{defn}
The positive maps $f,g$ in $\cW$ are \emph{positive trope equivalent}, denoted $g\s f$, if the maps in $\cP_F$ induced by $f$ and $g$ are positive trope equivalent.
\end{defn}
\begin{theorem}\label{differentbasepoints}
Given a flow space $\mX\approx \eX$, any flow space $\mY$ with expansion $\mY\approx \underleftarrow\lim(X_i,g_i,q_i)$ for which $g_i\s f_i$ for all $i$ is flow equivalent to $\mX.$    
\end{theorem}

\begin{proof}
Consider the following zigzag diagram of maps relating $\Pi\left( (X_i,f_i)\right)$ and $\Pi\left( (X_i,g_i)\right),$ where we suppress the notation for base points. 
\[
\begin{tikzcd}[font=\large] 
\pi_1(X_{1})  \arrow[swap]{d}{\text{id}} && \pi_1(X_{2})\arrow[swap]{d}{\text{id}}\arrow{ll}{\pi_1(f_{1})}&& \pi_1(X_{3})\arrow{ll}{\pi_1(f_{2})}\arrow[swap]{d}{\text{id}} &&  \cdots \arrow{ll}  \\ \pi_1(X_{1}) && \pi_1(X_{2}) \arrow{ll}{\pi_1(g_{1})}\arrow{llu}{\pi_1(f_1)} && \pi_1(X_{3})\arrow{ll}\arrow{ll}{\pi_1(g_{2})}\arrow{llu}{\pi_1(f_2)}  &&\cdots \arrow{ll}   
\end{tikzcd}
\]
This is a czz by construction, and so by Theorem \ref{ThmA} $\mX$ and $\mY$ are flow equivalent.
\end{proof}

In order to systematically construct expansions of a flow space about points from the whole range of path components, the following notion is likely crucial.

\begin{defn}\label{taileq}
For a category $\cC,$ two sequences $(C_i,f_i)$ and $(D_i,g_i)$ in $\overline{\cC}$ are \emph{tail equivalent}, denoted $(C_i,f_i) \sim_{tail} (D_i,g_i)$ if after possible truncations and re-indexing the two sequences are identical.
\end{defn}

Clearly, $ \sim_{tail}$ is an equivalence relation on $\cC$, and any two tail equivalent sequences in $\overline{\cC}$ are isomorphic in $\overline{\cC}$.

\begin{question}
Given an expansion of a flow space $\mX\approx \eX$, when is it possible to identify sets $\cF_i$ of maps positive trope equivalent to $f_i$ so that one can identify the orbits of the action of the homeomorphism group on $\mX$ with  
\[
\{\,(g_i)_{i\in \mZ^+}\colon g_i\in \cF_i \, \}/\sim_{tail}\,?   
\]
\end{question}

Observe that in the case of an expansion of a solenoid for which $X_i$ is the circle for each $i$, the above set consists of a single point, which matches the number of orbits of a solenoid's homeomorphism group.

In Theorem \ref{varied base point} we will see how to apply Theorem \ref{differentbasepoints} and afterwards indicate how the base points can be identified in some alternate expansions.
 
We now relate these results to homeomorphisms of Cantor sets. For a homeomorphism $h$ of a Cantor set, $PT(h)$ denotes the positive trope class of the suspension of $h.$ Thus, following Theorem \ref{germ}, we directly obtain  

\begin{theorem}
Two minimal homeomorphisms $h_1$ and $h_2$ of the Cantor set are germinally equivalent if and only if $PT(h_1)=PT(h_2).$
\end{theorem}

We have that the isomorphism classes in $\overline{\cW}$ which correspond to flow equivalence classes of minimal, aperiodic flow spaces are sent by $\Pi$ to positive trope equivalence classes in $\overline{\cP_f}$. Coarser functors will still provide an invariant of flow equivalence but will not be a complete invariant if strictly coarser. Nevertheless, they  can be useful when more easily computed as we see for the next invariant, demonstrated in an application in Section \ref{Applications}.

\begin{defn}
Let $\cP_{\cA}$ denote the category of ordered pairs $(A_n, \cG)$ consisting of a finitely generated free abelian group $A_n$ and a chosen set of generators $\cG.$ The morphisms of  $\cP_{\cA}$ are the positive homomorphisms: those that preserve the semigroups generated by the chosen generators. 
\end{defn}

\begin{defn}
For a sequence $(X_i,f_i)\in \overline{\cW}$, let  $H\left( (X_i,f_i) \right)$ be the sequence in $\overline{\cP_\cA}$ given by applying homology $H_1$ to the $X_i$ and $f_i$ and choosing generators corresponding to the circles in $X_i$ with the positive orientation.
\end{defn}

As abelianisation trivialises conjugations, considering the corresponding facts for $\Pi$, the following is well-defined and independent of both expansion and  base point.

\begin{defn}
For a minimal, aperiodic flow space $\mX$ define  $\h(\mX)$ to be the isomorphism class in $\overline{\cP_\cA}$ of $H\left( (X_i,f_i) \right)$  for any $(X_i,f_i)$ corresponding to an expansion $(\mX,x_\bullet)\approx \eX$.
\end{defn}

Specifically, if $PT(\mX)=PT(\mY)$, then $\h(\mX)=\h(\mY)$. This then immediately leads to the following result.

\begin{theorem}
If the minimal, aperiodic flow spaces $\mX$ and $\mY$ are flow equivalent, then $\h(\mX)=\h(\mY)$.
\end{theorem}

The invariant given by the homology core developed in \cite{CH} for spaces of all dimensions can be be derived from $\h$ for the one-dimensional flow spaces. 

A similar construction can be made with cohomology. 







\section{Applications}\label{Applications}

We will now see how to apply $\p$, $PT$, and $\h$ to classify various families of flow spaces. 

We start with a general result. We call a positive homomorphism in $\cP_F$ \emph{invertible} if it has a, possibly non-positive, inverse as a group homomorphism.

\begin{theorem}\label{invertible}
Suppose $\sigma$ and $\tau$ are proper, primitive, positive homomorphisms of $F_n$, and let $f_\sigma$ and $f_\tau$ be maps of $W_n$ which $\sigma$ and $\tau$ represent. Let $\Omega_\sigma$ and $\Omega_\tau$ be the flow spaces given by the limits of the stationary sequences $(X_i,f_i)$, where $X_i=W_n$ and $f_i=f_*$ for all $i,$ where $*$ is $\sigma$ for $\Omega_\sigma$ and is $\tau$ for $\Omega_\tau$. If $\sigma$ is invertible and $\tau$ is non-invertible, then $\Omega_\sigma$ and $\Omega_\tau$ are not flow equivalent.
\end{theorem}

\begin{proof}
Recall that a group $G$ is {\em Hopfian\/} if any endomorphism $f\colon G\to G$ which is onto is an isomorphism. It is known that $F_n$ is Hopfian.

Suppose under our hypotheses that $\Omega_\sigma$ and $\Omega_\tau$ were flow equivalent. Then as $PT(\Omega_\sigma)=PT(\Omega_\tau)$, there are positive homomorphisms $u,v\colon F_n\to F_n$ such that, for some $\ell$ and $m$, $\sigma^\ell$ agrees with $v\circ\tau^m\circ u$ up to conjugation by some element $h\in F_n$; i.e., $h^{-1}(\sigma^\ell(x))h=v\circ\tau^m\circ u(x)$. 

As $\sigma$ is invertible, $\sigma^\ell$ and hence $v\circ\tau^m\circ u$ are invertible since conjugation is a group isomorphism. In particular, $v$ is onto and hence is a group isomorphism by the Hopfian property. As $\tau$ is not invertible, $\tau^m$ is not onto. So the image of $\tau^m\circ u$ is a strict subgroup of $F_n$. As $v$ is a group isomorphism, the  image of $v\circ\tau^m\circ u$ is a strict subgroup of $F_n$, contradicting the earlier point that $v\circ\tau^m\circ u$ is an isomorphism. 
\end{proof}
The spaces $\Omega_\sigma$ and $\Omega_\tau$ are examples of substitution tiling spaces as discussed in Section \ref{expansionexamp}.

\begin{ex}\label{twistedfib}

Consider the following positive homomorphisms $F_2 \to F_2$
$$\sigma:\ a\mapsto aabaabab\quad b\mapsto aabab\qquad\qquad\tau: a\mapsto aabaabab\quad b\mapsto abaab\,.
$$

Using the folding lemma of Stallings, we see that $\sigma$ is invertible but $\tau$ is not. Thus, $\Omega_\sigma$ and $\Omega_\tau$ are not flow equivalent. Observe that the matrices representing the abelianisations of $\sigma$ and $\tau$ are identical, from which it follows that the corresponding homeomorphisms of Cantor sets are orbit equivalent, \cite{DHS}. From this it also follows that $\h(\Omega_\sigma)=\h(\Omega_\tau)$ and that the spaces have isomorphic Cech cohomology. 
\end{ex}

The ideas of Theorem \ref{invertible} can be extended to cover examples such as s-adic sequences.

\begin{theorem}\label{geninvertible}
Suppose $\mX$ and $\mY$ are minimal, aperiodic flow spaces with expansions $\mX\approx \eX$ and $\mY \approx \eY$ for which each $X_i$ and $Y_i$ is $W_n$. If for each $i,$ $\pi_1(f_i)$ is invertible and if for each $i$ there is a $j>i$ for which $\pi_1(g_j)$ is not invertible, then  $\mX$ and $\mY$ are not flow equivalent.
\end{theorem}

\begin{proof}
Suppose  $PT(\mX)=PT(\mY)$ and consider the czz equating the positive trope classes of $\Pi\left( (X_i,f_i)\right)$ and  $\Pi\left( (Y_i,g_i)\right)$. Then for some $\ell_1,\ell_2$ and $m_1,m_2$, there are positive homomorphisms $u,v\colon F_n\to F_n$ such that, the invertible homomorphism $\pi_1(f_{\ell_1,\ell_2})$ agrees with $v\circ\pi_1(g_{m_1,m_2})\circ u$ up to conjugation by some element $h\in F_n$, where $\pi_1(g_{m_1,m_2})$ is not invertible. By the same argument as in the proof of Theorem \ref{invertible}, this leads to a contradiction.
\end{proof}

We now present examples of two different families of flow spaces, each with an uncountable number of flow equivalence classes and which have only one flow equivalence class in common. For this we will make use of the $\sigma$ and $\tau$ of Example \ref{twistedfib}  together with the invertible positive homomorphism $F_2 \to F_2$ given by 
$$\rho\colon a\mapsto ba\quad b\mapsto bba\,,$$ 
which represents the positive map $f_\rho \colon W_2 \to W_2 $. 

Let $\Sigma$ be the family of flow spaces $\mX$ admitting an expansion $\mX \approx \eX$ for which each $X_i=W_2$ and each $f_i$ is either $f_\sigma$ or $f_\rho$. Similarly, let $\rm{T}$ be the family of flow spaces $\mX$ admitting an expansion $\mX \approx \eX$ for which each $X_i=W_2$ and each $f_i$ is either $f_\tau$ or $f_\rho$. 

For the classification, the following notion will be useful. 
As our invariants are independent of the expansion used for a flow space, we are justified in only considering the expansions of the flow spaces in $\Sigma$ and $\rm{T}$ that use $f_\rho$, $f_\sigma$  and $f_\tau$. We will refer to such an expansion as an \emph{admissible} expansion, and we will identify these expansions with the corresponding sequences in the symbols $\sigma$, $\tau$ and $\rho.$  Also, we refer to a flow space  in $\Sigma$ or $\rm{T}$ as \emph{non-constant} if it has an admissible expansion which is not tail equivalent to a constant (stationary) sequence.

It is not difficult to see that the constant flow spaces $\Omega_\sigma$ and $\Omega_\rho$ are flow equivalent. They are in fact flow equivalent to the Fibonacci substitution tiling space. We have already established that $\Omega_\tau$ is not flow equivalent to $\Omega_\sigma$.

We have as a direct consequence of Theorem \ref{geninvertible} the following result.

\begin{theorem}
No non-constant flow space in $\Sigma$ is flow equivalent to a non-constant flow space in $\rm{T}$.
\end{theorem}

Observe that this is despite the fact that for each non-constant flow space in $\Sigma$, replacing each $\sigma$ in an admissible expansion with a $\tau$ yields a space in $\rm{T}$ for which all abelian invariants coincide. In fact, we shall now see that $\h$ is sufficient to classify the spaces within either family, and so this allows us to classify the spaces within $\Sigma$ and $\rm{T}$ simultaneously.
For this, we will need the following result, which is obtained in \cite[Lemma 4.2]{BW} with the aid of the Farey array.

\begin{lemma}\label{uniquefact}
Any non-negative element $M\in GL_2(\mZ)$ has a {\bf unique} factorization 
$$M\ =\ M_0M_1\cdots M_n$$
where 
$$M_0\in\left\{\left(\begin{array}{cc}1&0\\0&1\\\end{array}\right),\ \left(\begin{array}{cc}0&1\\1&0\\\end{array}\right)\right\}
\qquad\qquad M_i\in\left\{\left(\begin{array}{cc}1&1\\0&1\\\end{array}\right),\ \left(\begin{array}{cc}1&0\\1&1\\\end{array}\right)\right\},\ i\not=0\,.$$
\end{lemma}

\begin{theorem}
Two non-constant flow spaces in $\Sigma$ or $\rm{T}$ are flow equivalent if and only if their admissible expansions are tail equivalent. 
\end{theorem}
\begin{proof}
If two flow spaces have expansions are tail equivalent as in Definition \ref{taileq}, they are clearly flow equivalent. 

Let now $\mX$ and $\mY$ be two flow equivalent non-constant flow spaces from one of the families. Consider the commutative diagram representing the isomorphism of $\h(\mX)$ and $\h(\mY)$ resulting from the flow equivalence relative to admissible expansions:
\[
\begin{tikzcd}[font=\large] 
H_1\left(X_{m_1} \right) \arrow[swap]{d}{\left( d_1\right)_* } && H_1\left(X_{m_2}\right)\arrow[swap]{d}{\left( d_2\right)_*}\arrow{ll}{\left( f_{m_1,\,m_2}\right)_*}&& H_1\left(X_{m_3}\right)\arrow{ll}{\left( f_{m_2,\,m_3}\right)_*}\arrow[swap]{d}{\left( d_3\right)_*} &&  \cdots \arrow{ll}  \\ H_1\left(Y_{n_1}\right) && H_1\left(Y_{n_2} \right)\arrow{ll}{\left( g_{n_1,\,n_2}\right)_*}\arrow{llu}{\left( u_1\right)_*} && H_1\left(Y_{n_3} \right)\arrow{ll}\arrow{ll}{\left( g_{n_2,\,n_3}\right)_*}\arrow{llu}{\left( u_2\right)_*}  &&\cdots \arrow{ll}   
\end{tikzcd}
\]
The matrix representing the maps induced by $\sigma$ and $\tau$ is factored as
\[
\left(\begin{array}{cc}1&1\\0&1\\\end{array}\right)\left(\begin{array}{cc}1&0\\1&1\\\end{array}\right)\left(\begin{array}{cc}1&1\\0&1\\\end{array}\right)\left(\begin{array}{cc}1&0\\1&1\\\end{array}\right)\, ,
\]
whereas the matrix representing $\rho$ is factored as $\left(\begin{array}{cc}1&0\\1&1\\\end{array}\right)\left(\begin{array}{cc}1&1\\0&1\\\end{array}\right)$. From the diagram, for each $i$ we have $\left( d_1\right)_*\circ \left( f_{m_1,\,m_i}\right)_* = \left( g_{n_1,\,n_i}\right)_*\circ \left( d_i\right)_*$, and factoring the resulting equal matrices as in  Lemma \ref{uniquefact}, we can identify the occurrence of each $\rho$ in the corresponding expansion by the adjacent identical matrices, after perhaps some truncation depending on  $\left( d_1\right)_*$.  As this equation of matrices holds for all $i$, we can identify the corresponding increasing portions of the expansions. Thus, the expansions for $\mX$ and $\mY$ are tail equivalent, truncating according to the values $m_1$ and $n_1$ and the matrix representing $\left( d_1\right)_*$.
\end{proof}

Thus, each of the families has uncountably many flow equivalence classes and the only overlap between the two families is in the class of the constant flow spaces $\Omega_\sigma$ and $\Omega_\rho$, while the constant $\Omega_\tau$ is only flow equivalent to its tail equivalents.

The family $\Sigma$ is a subfamily of the suspensions of Sturmian sequences, the Denjoy continua. In fact our technique can be readily used to classify all the Denjoy continua, which had been previously established in \cite{BW},\cite{F}.  

The family $\Sigma$ illustrates the importance of our requirement that the homomorphisms in the construction of $\p$ and $PT$ are positive, as without this requirement the admissible expansions of \emph{any} two spaces from the Sturmian family admit a commutative zigzag diagram of induced maps in $\pi_1$ since all the homomorphisms are invertible, but this is only by allowing the homomorphisms $d_n$ and $u_n$ not to be positive: the key to this observation is essentially that the inverse of a positive isomorphism is not generally positive. (For example, the inverse to the positive isomorphism $a\mapsto aab,\quad b\mapsto ab$ is $a\mapsto ab^{-1},\quad b\mapsto ba^{-1}b$.) From this it follows that these spaces are all shape equivalent and have isomorphic pro-$\pi_1$. Thus we see our invariants are significantly finer than those defined just from the shape category. In  particular, it shows the difference between the positive trope and the fundamental trope of Fox, which was developed to define shape theory.  

Interestingly, in $\Sigma$ and $\rm{T}$ flow equivalence coincides precisely with pointed flow equivalence relative to the admissible expansions as tail equivalence preserves base points and results in isomorphic $\p$. In contrast, consider the following positive endomorphisms of $F_2$: 

$$\alpha:\ a\mapsto abaababa\quad b\mapsto ababa\qquad\qquad\beta: a\mapsto baabaaba\quad b\mapsto baaba.
$$

With notation as above, we let $V$ be the family of flow spaces admitting an expansion using only $f_\sigma$, $f_\alpha$ and $f_\beta$.

\begin{theorem}\label{varied base point}
 All flow spaces in $V$ are flow equivalent  to $\Omega_\sigma$.
\end{theorem}
\begin{proof}
As $\alpha=\cc_a \sigma $ and $\beta= \cc_{b^{-1}} \sigma$, this follows directly from \ref{differentbasepoints}. 
\end{proof}
The expansions in $V$ can be viewed as alternative expansions of $\Omega_\sigma$ with different choices of base point. To see this, consider the flow equivalence $d$ constructed as in the proof of Theorem \ref{ThmA} from the czz in the proof of Theorem \ref{differentbasepoints}, where we view the top row of the diagram to have all $\sigma$'s as the horizontal maps and the bottom row to have horizontal maps according to the admissible expansion of a chosen element of $V$. The corresponding subshifts associated to the initial tiling space factors are identical to the substitution subshift  $S(\sigma)$ as in Definition \ref{substsubshiftdefn}. Following through the construction, and using the notation of Section \ref{expansionexamp}, we see that the base point for $\Omega_\alpha$ considered as a point in the  the base $S(\sigma)$ is  $\alpha^\infty(a)\,\cdot \alpha^\infty(a)$, whereas the base point for $\Omega_\beta$ is $\beta^\infty(a)\,\cdot\, \beta^\infty(b)$. For the more general element of $V$, the base point will be determined by the limit sequence obtained by successive applications of the corresponding substitutions. 

The phenomenon that leads to the classification of the non-constant flow spaces within $\Sigma$ and $\rm{T}$ can be directly generalised as follows, where we use analogous terminology and notation to the above.

\begin{theorem}
Suppose $\sigma_1,\ldots, \sigma_m$ are primitive, proper positive endomorphisms of $F_n$ which are represented in homology by the square matrices $A_i\in M_n(\mZ^+)$. Suppose the $A_i$ {\em freely\/} generate the semi-group $A=\langle A_1,\ldots,A_n\rangle \subset M_m(\mZ^+)$. With $S$ denoting the collection of all flow spaces admitting expansions using the $f_{\sigma_i}$, flow spaces in $S$ are flow equivalent if and only if their admissible expansions are tail equivalent. 
\end{theorem}
The freeness of the generation is to be interpreted as no two distinct words in the $A_i$ yield equal matrix products.

\section{Concluding remarks}
We now see how we can identify the flow equivalence classes with naturally defined algebraic structures. While the notion of positive trope equivalence might at first seem abstract and unwieldy, the proof of its completeness as an invariant reveals that it is in fact quite approachable, as illustrated in Theorem \ref{differentbasepoints}. In particular, the conjugations involved in a czz diagram do not alter the subshifts associated to the expansions of the corresponding flow spaces. In \cite{U} it is shown that a particular pair of substitution subshifts that differ only in the order of two symbols in their substitutions are not flow equivalent, as in Example \ref{twistedfib}. It is to be expected that by applying appropriate tools of combinatorial group theory our invariants can be used to classify many and various large families of flow spaces. 

It is natural to wonder to what extent the results we have developed here could be generalised to higher dimensions. The notion of germinal equivalence is related to the notion of return equivalence as developed in \cite{F} in one dimension and in \cite{CHL} for matchbox manifolds of higher dimension. As shown in \cite{CHL}, there is a limit even in the equicontinuous case to what can be determined about the global matchbox manifold from the dynamics on a transversal. This, coupled with the large variety of spaces beyond circles that could appear as building blocks of approximating spaces in expansions in higher dimensions, makes the task quite imposing. However, if one restricts attention to those spaces that support an action of $\mR^d$, there is reason to hope that significant generalisations are possible.

\end{document}